\documentclass[12pt]{amsart}
\usepackage[T1]{fontenc}
\usepackage[utf8]{inputenc}
\usepackage[
  backend=biber,
  style=numeric,
  sorting=none,
  giveninits=false,
  doi=true,
  isbn=true,
  url=true,
  eprint=true
]{biblatex}
\usepackage{amsmath, amssymb, amsthm}
\usepackage{mathrsfs}
\usepackage{bbm}
\usepackage{hyperref}
\usepackage{geometry}
\usepackage{enumitem}
\usepackage{booktabs}
\usepackage{array}
\usepackage{tabularx}
\usepackage{tikz-cd}   

\newtheorem{theorem}{Theorem}[section]
\newtheorem{proposition}[theorem]{Proposition}
\newtheorem{lemma}[theorem]{Lemma}
\newtheorem{corollary}[theorem]{Corollary}
\newtheorem{definition}[theorem]{Definition}

\theoremstyle{remark}
\newtheorem{remark}[theorem]{Remark}

\newcommand{\Ext}{\mathrm{Ext}}
\newcommand{\Hom}{\mathrm{Hom}}

\newcommand{\id}{\mathrm{id}}

\newcommand{\tors}{\mathrm{tors}}

\DeclareMathOperator{\coker}{coker}

\begin{document}
\title[Integral Perverse Obstructions for Surface Singularities]{Integral Perverse Obstructions for Normal Surface Singularities:
Resolution Determinants and Monodromy}

\author{Abdul Rahman}
\thanks{Email: arahman@alum.howard.edu}

\begin{abstract}
For a germ $(X,0)$ of a normal complex analytic surface, let
$E:=H^0({}^p_+IC_X\mathbb Z)_0$, where ${}^pIC_X\mathbb Z$ and ${}^p_+IC_X\mathbb Z$ denote the ordinary and dual middle-perversity intersection complexes with integral coefficients. This finite abelian group measures the integral discrepancy between the two
middle extensions. Motivated by work of Jung--Saito, we study $E$ as a local invariant of the singularity. We prove that $E$ admits a topological realization as $H^2(L,\mathbb Z)_{\tors}$, where $L$ is the link of the singularity, and a geometric realization as the discriminant group of the exceptional lattice of the minimal
resolution. In particular, if $M$ is the intersection matrix of the irreducible exceptional curves, then $|E|=|\det(M)|$. If $(X,0)$ is an isolated hypersurface surface singularity, we further prove that $E\cong \coker(T-\id)_{\tors}$, where $T$ is the Milnor monodromy
on integral vanishing cohomology. Under the additional hypothesis that
$(T-\id)\otimes_{\mathbb Z}\mathbb Q$ is an isomorphism, this yields
$|E|=|\det(T-\id)|$. Thus the same local integral obstruction admits compatible perverse, topological, resolution-theoretic, and monodromy-theoretic realizations.
\end{abstract}
\maketitle
\tableofcontents

\section{Introduction}

Let $(X,0)$ be a germ of a normal complex analytic surface. In the local surface case,
Jung and Saito define the finite abelian group
\[
E:=H^0({}^p_+IC_X\mathbb Z)_0,
\]
and record a self-dual distinguished triangle
\[
{}^pIC_X\mathbb Z \longrightarrow {}^p_+IC_X\mathbb Z \longrightarrow E[1]\longrightarrow .
\]
They further indicate that this local correction should be governed by the discriminant of
the exceptional intersection form of the minimal resolution and, in the hypersurface case,
by the integral variation map $T-\id$ on vanishing cohomology; see
\cite[Remark~6.4]{JungSaitoFactoriality}. The purpose of this paper is to make these
indications precise.

The discrepancy between ordinary middle perversity and dual middle perversity disappears
after tensoring with $\mathbb Q$, but it may persist integrally. For a normal surface germ,
this discrepancy is concentrated at the singular point and is measured by the finite group
$E$. The main result of the paper is that $E$ admits the following compatible realizations:
\[
E
\quad\cong\quad
H^2(L,\mathbb Z)_{\tors}
\quad\cong\quad
\Lambda^\vee/\Lambda
\quad\cong\quad
\coker(T-\id)_{\tors},
\]
where $L$ is the link of $(X,0)$, $\Lambda$ is the exceptional lattice of the minimal
resolution, and $T$ is the Milnor monodromy in the isolated hypersurface case. Under the
additional hypothesis that $(T-\id)\otimes_{\mathbb Z}\mathbb Q$ is an isomorphism, one
obtains the determinant formula
\[
|E|=|\det(T-\id)|.
\]

Thus the same local correction is visible in four different ways: as the discrepancy
between the two integral middle extensions, as torsion in the second cohomology of the
link, as the discriminant of the exceptional lattice, and, for isolated hypersurface
surface singularities, as torsion in the cokernel of the integral variation map. In this
sense, $E$ is a local integral discriminant invariant of the singularity.

We now state the main results.

\begin{theorem}[Integral perverse obstruction]
Let $(X,0)$ be a germ of a normal complex analytic surface, and define
\[
E:=H^0({}^p_+IC_X\mathbb Z)_0.
\]
Then the natural morphism
\[
{}^pIC_X\mathbb Z \longrightarrow {}^p_+IC_X\mathbb Z
\]
has cone supported at $0$, and there is a self-dual distinguished triangle
\[
{}^pIC_X\mathbb Z \longrightarrow {}^p_+IC_X\mathbb Z \longrightarrow E[1]\longrightarrow .
\]
Moreover, $E$ is a finite abelian group depending only on the analytic germ $(X,0)$.
\end{theorem}

\begin{theorem}[Resolution-lattice realization]
Let
\[
\pi:\widetilde X\to X
\]
be the minimal resolution of $(X,0)$, let $\Lambda$ be the lattice generated by the
irreducible exceptional curves, and let $M$ be the corresponding intersection matrix.
Then
\[
E\cong \Lambda^\vee/\Lambda.
\]
In particular,
\[
|E|=|\det(M)|.
\]
\end{theorem}

\begin{theorem}[Hypersurface-monodromy realization]
Assume that $(X,0)$ is an isolated hypersurface surface singularity. Then
\[
E\cong \coker(T-\id)_{\tors}.
\]
\end{theorem}

\begin{corollary}[Determinant refinement]
Assume, in addition, that
\[
(T-\id)\otimes_{\mathbb Z}\mathbb Q
\]
is an isomorphism on
$H^2_{\mathrm{van}}(F,\mathbb Z)\otimes_{\mathbb Z}\mathbb Q$. Then
\[
E\cong \coker(T-\id),
\qquad
|E|=|\det(T-\id)|.
\]
\end{corollary}

\begin{corollary}[Rational double points]
For rational double points of type $A$, $D$, and $E$, the perverse, topological,
resolution-theoretic, and monodromy-theoretic realizations of $E$ are compatible. In
type $A_k$, one has
\[
E\cong \mathbb Z/(k+1)\mathbb Z,
\qquad
|E|=k+1.
\]
\end{corollary}

The proof proceeds by passing from the local perverse correction to topology, from the
link to the exceptional lattice, and, in the hypersurface case, from link cohomology to
Milnor monodromy.

First, in Section~2, we isolate the local correction term from the discrepancy between \({}^pIC_X\mathbb Z\) and \({}^p_+IC_X\mathbb Z\) and prove a formal cohomological range statement for the discrepancy cone.  In Appendix~\ref{app:rational-codim2}, we prove the rational codimension-two gluing theorem by identifying the relevant Friedman
torsion-sensitive perverse hearts \cite{FriedmanGenIH} with the BBD \cite{BBD82} ordinary and dual integral middle-perversity hearts.  For the full general codimension-two concentration and self-duality, we still rely on the local surface statement of Jung--Saito \cite{JungSaitoFactoriality}.

Second, in Section~3, we identify $E$ with the torsion subgroup of the second cohomology
of the link:
\[
E\cong H^2(L,\mathbb Z)_{\tors}.
\]

Third, in Section~4, we pass to the minimal resolution and prove
\[
H^2(L,\mathbb Z)_{\tors}\cong \Lambda^\vee/\Lambda.
\]
This yields the resolution-theoretic realization of $E$ and the formula
\[
|E|=|\det(M)|.
\]

Fourth, in Section~5, for isolated hypersurface surface singularities, the Wang sequence
of the Milnor fibration identifies
\[
H^2(L,\mathbb Z)\cong \coker(T-\id).
\]
Combining this with the topological description of $E$ gives
\[
E\cong \coker(T-\id)_{\tors}.
\]
Under the additional rational-invertibility hypothesis, the whole cokernel is finite, and
one obtains
\[
|E|=|\det(T-\id)|.
\]

Finally, in Section~6, we work out explicit examples, including the ADE singularities and
a non-ADE Brieskorn--Pham family, to illustrate the resulting dictionary.

The paper is organized accordingly. Section~2 treats the local perverse package.
Section~3 gives the topological realization of $E$ through the link. Section~4 identifies
the same group with the discriminant group of the exceptional lattice. Section~5 gives
the monodromy-theoretic realization in the isolated hypersurface case. Section~6 contains
examples. The final sections explain how this local invariant fits into the broader
Jung--Saito picture surrounding factoriality and dual middle perversity.


\section{Integral middle perversities and the local obstruction group}

In this section we set up the two integral middle-perversity packages, define the local group
$E$, and analyze the point-supported correction term measuring the discrepancy between the
ordinary and dual middle extensions in the case of a normal surface germ. The point is to
separate four issues: the formal support statement, the local model provided by
$Rj_*\mathbb Z_U[2]$, the cohomological range forced by the point-stratum conditions, and
the final codimension-two concentration statement.

Throughout this paper, all spaces are complex analytic, and all derived categories are
bounded constructible derived categories with $\mathbb Z$-coefficients unless otherwise
specified. For a complex analytic space $X$, we write $D^b_c(X,\mathbb Z)$ for the
corresponding constructible derived category. We use the standard notation
$i_x:\{x\}\hookrightarrow X$ for the inclusion of a point and $j:U\hookrightarrow X$
for the inclusion of a Zariski-open dense smooth subset. Verdier duality is denoted by
$\mathbb D$.

\subsection{Middle perversity and dual middle perversity over \texorpdfstring{$\mathbb Z$}{Z}}

We begin by recalling the ordinary middle-perversity and dual middle-perversity
$t$-structures on $D^b_c(X,\mathbb Z)$ following \cite{BBD82}. For
$K\in D^b_c(X,\mathbb Z)$, one sets
\[
{}^pD^{\le 0}(X,\mathbb Z)
:=
\Bigl\{
K\in D^b_c(X,\mathbb Z)\ \Big|\
H^k(i_x^*K)=0 \ \text{for all }x\in X,\ k>-\dim_{\mathbb C}\overline{\{x\}}
\Bigr\},
\]
and
\[
{}^pD^{\ge 0}(X,\mathbb Z)
:=
\Bigl\{
K\in D^b_c(X,\mathbb Z)\ \Big|\
H^k(i_x^!K)=0 \ \text{for all }x\in X,\ k<-\dim_{\mathbb C}\overline{\{x\}}
\Bigr\}.
\]
Their heart is the abelian category of perverse sheaves
\[
{}^p\!\operatorname{Perv}(X,\mathbb Z)
:=
{}^pD^{\le 0}(X,\mathbb Z)\cap {}^pD^{\ge 0}(X,\mathbb Z).
\]

With integral coefficients, Verdier duality does not preserve the ordinary middle perversity
in general because torsion may appear in stalk cohomology. One therefore introduces the
dual middle-perversity $t$-structure, denoted here by
${}^p_+D^{\le 0}$ and ${}^p_+D^{\ge 0}$, characterized by
\[
K\in {}^p_+D^{\le 0}(X,\mathbb Z)
\quad\Longleftrightarrow\quad
\mathbb D K \in {}^pD^{\ge 0}(X,\mathbb Z),
\]
and
\[
K\in {}^p_+D^{\ge 0}(X,\mathbb Z)
\quad\Longleftrightarrow\quad
\mathbb D K \in {}^pD^{\le 0}(X,\mathbb Z).
\]
Equivalently, ${}^p_+$ is obtained from ${}^p$ by exchanging torsion and torsion-free
conditions in the critical degree. We shall use the explicit semi-perversity conditions
recorded by Jung--Saito in Remark~6.3: for any Whitney stratum $S$ of complex dimension
$d_S$ with inclusion $i_S:S\hookrightarrow X$,
\[
F^\bullet\in {}^pD^{\le j}
\iff H^k(i_S^*F^\bullet)=0 \ \text{for } k>j-d_S,
\]
\[
F^\bullet\in {}^pD^{\ge j}
\iff H^k(i_S^!F^\bullet)=0 \ \text{for } k<j-d_S,
\]
while the dual semi-perversity conditions are
\[
F^\bullet\in {}^p_+D^{\le j}
\iff
\begin{cases}
H^k(i_S^*F^\bullet)=0 & \text{for } k>j-d_S+1,\\
H^{j-d_S+1}(i_S^*F^\bullet)\ \text{torsion},
\end{cases}
\]
and
\[
F^\bullet\in {}^p_+D^{\ge j}
\iff
\begin{cases}
H^k(i_S^!F^\bullet)=0 & \text{for } k<j-d_S,\\
H^{j-d_S}(i_S^!F^\bullet)\ \text{torsion-free}.
\end{cases}
\]
See \cite[Remark~6.3, formulas (6.9)--(6.10)]{JungSaitoFactoriality}. These are the
pointwise conditions used below.

The key point for us is that the two hearts agree after tensoring with $\mathbb Q$, but they
need not agree over $\mathbb Z$. Thus the difference between
${}^p\!\operatorname{Perv}(X,\mathbb Z)$ and
${}^p_+\!\operatorname{Perv}(X,\mathbb Z)$
is a purely integral phenomenon controlled by torsion in local stalk and costalk cohomology.

\subsection{The complexes \texorpdfstring{${}^pIC_X\mathbb Z$}{pIC} and \texorpdfstring{${}^p_+IC_X\mathbb Z$}{p+IC}}

Let $X$ be a reduced irreducible complex analytic space of pure dimension $n$, and let
$j:X_{\mathrm{reg}}\hookrightarrow X$ be the inclusion of the smooth locus. The ordinary
middle-perversity intersection complex with $\mathbb Z$-coefficients is
\[
{}^pIC_X\mathbb Z
:=
{}^pj_{!*}\,\mathbb Z_{X_{\mathrm{reg}}}[n],
\]
where ${}^pj_{!*}$ denotes intermediate extension in the ordinary middle-perversity heart.
Likewise, the dual middle-perversity intersection complex is
\[
{}^p_+IC_X\mathbb Z
:=
{}^p_+j_{!*}\,\mathbb Z_{X_{\mathrm{reg}}}[n].
\]

Both complexes extend the shifted constant local system from the smooth locus:
\[
j^*({}^pIC_X\mathbb Z)\cong \mathbb Z_{X_{\mathrm{reg}}}[n],
\qquad
j^*({}^p_+IC_X\mathbb Z)\cong \mathbb Z_{X_{\mathrm{reg}}}[n].
\]
Hence any discrepancy between them is necessarily supported on the singular set.

When $X$ is normal of complex dimension $2$, the singular locus is discrete. Therefore, for
a germ $(X,0)$ of a normal surface singularity, after shrinking $X$ we may assume that $0$
is the unique singular point and that the difference between ${}^pIC_X\mathbb Z$ and
${}^p_+IC_X\mathbb Z$ is supported at $0$.

We shall use repeatedly that both constructions are functorial with respect to analytic
isomorphisms of germs. Thus any local invariant extracted from
${}^pIC_X\mathbb Z$ and ${}^p_+IC_X\mathbb Z$ depends only on the analytic type of the germ.

\subsection{Definition of the local group \texorpdfstring{$E$}{E}}

From now on, let $(X,0)$ be a germ of a normal complex analytic surface. Following
Jung--Saito, we define
\[
E:=H^0({}^p_+IC_X\mathbb Z)_0,
\]
that is, the degree-zero stalk cohomology of the dual middle-perversity intersection complex
at the singular point. This is exactly the group introduced in
\cite[Remark~6.4]{JungSaitoFactoriality}.

Since the two intersection complexes agree on $X\setminus\{0\}$, the group $E$ measures the
failure of the two integral middle extensions to coincide at the singular point. It is
therefore the local integral correction term attached to the discrepancy between the ordinary
and dual middle-perversity packages.

We emphasize that this correction is invisible after tensoring with $\mathbb Q$. Indeed, the
ordinary and dual intersection complexes become canonically isomorphic over $\mathbb Q$, so
$E$ is a purely torsion-theoretic invariant.

\subsection{The self-dual distinguished triangle}

We now analyze the local discrepancy between the two middle extensions. The first formal
point is that the discrepancy is point-supported. The second is that the point-stratum
semi-perversity bounds force its cohomological amplitude to be very small. The third is
that the local model for the extension problem is
\[
K:=Rj_*\mathbb Z_U[2],
\qquad U:=X\setminus\{0\}.
\]

\begin{lemma}\label{lem:cone-supported-at-zero}
Let $(X,0)$ be a germ of a normal complex analytic surface, and let
\[
u:{}^pIC_X\mathbb Z \longrightarrow {}^p_+IC_X\mathbb Z
\]
be the natural morphism induced by the identity on the smooth locus. Then the cone
\[
C^\bullet:=\operatorname{Cone}(u)
\]
is supported at $\{0\}$.
\end{lemma}

\begin{proof}
Both ${}^pIC_X\mathbb Z$ and ${}^p_+IC_X\mathbb Z$ restrict to the same shifted local
system $\mathbb Z_{X_{\mathrm{reg}}}[2]$ on the smooth locus. Hence
\[
j^*u:j^*({}^pIC_X\mathbb Z)\xrightarrow{\sim} j^*({}^p_+IC_X\mathbb Z)
\]
is an isomorphism on $X\setminus\{0\}$. Therefore
\[
j^*C^\bullet \cong 0,
\]
so the support of $C^\bullet$ is contained in $\{0\}$.
\end{proof}

\begin{lemma}\label{lem:point-supported-star-shriek}
Let $i=i_0:\{0\}\hookrightarrow X$, and let $K^\bullet\in D^b_c(X,\mathbb Z)$ be supported
on $\{0\}$. Then the natural morphisms
\[
i_*i^*K^\bullet \longrightarrow K^\bullet,
\qquad
K^\bullet \longrightarrow i_*i^!K^\bullet
\]
are isomorphisms. In particular,
\[
i^*K^\bullet \cong i^!K^\bullet.
\]
\end{lemma}

\begin{proof}
Since $K^\bullet$ vanishes on $X\setminus\{0\}$, it lies in the essential image of the fully
faithful functor
\[
i_*:D^b_c(\{0\},\mathbb Z)\longrightarrow D^b_c(X,\mathbb Z).
\]
Hence $K^\bullet\cong i_*A^\bullet$ for the bounded complex $A^\bullet:=i^*K^\bullet$, which
gives the first isomorphism. The second is proved similarly using $i^!i_*\cong \mathrm{id}$.
\end{proof}

\begin{lemma}\label{lem:stalks-of-K}
Let
\[
K:=Rj_*\mathbb Z_U[2],
\qquad U:=X\setminus\{0\}.
\]
Then there are canonical isomorphisms
\[
H^m(i_0^*K)\cong H^{m+2}(L,\mathbb Z)
\qquad \text{for all }m\in\mathbb Z.
\]
In particular,
\[
H^{-2}(i_0^*K)\cong \mathbb Z,\qquad
H^{-1}(i_0^*K)\cong H^1(L,\mathbb Z),\qquad
H^0(i_0^*K)\cong H^2(L,\mathbb Z),\qquad
H^1(i_0^*K)\cong \mathbb Z,
\]
and $H^m(i_0^*K)=0$ for $m\notin\{-2,-1,0,1\}$.
\end{lemma}

\begin{proof}
By definition of derived pushforward,
\[
H^m(i_0^*K)
=
H^m(i_0^*Rj_*\mathbb Z_U[2])
\cong
(R^{m+2}j_*\mathbb Z_U)_0.
\]
The stalk of $R^qj_*\mathbb Z_U$ at $0$ is the cohomology of a sufficiently small punctured
neighborhood of $0$ in $U$, hence of $X_\varepsilon\setminus\{0\}$. Since
\[
X_\varepsilon\setminus\{0\}\simeq L,
\]
one obtains
\[
(R^{m+2}j_*\mathbb Z_U)_0\cong H^{m+2}(L,\mathbb Z).
\]
This proves the general formula. The explicit list follows because $L$ is a connected
oriented closed real $3$-manifold.
\end{proof}

The next lemma records the point-stratum stalk and costalk bounds coming from the
semi-perversity conditions of Remark~6.3 of Jung--Saito \cite{JungSaitoFactoriality}.

\begin{lemma}\label{lem:point-stratum-bounds}
Let $(X,0)$ be a germ of a normal complex analytic surface, and let $i=i_0:\{0\}\hookrightarrow X$.
Set
\[
A:={}^{p}IC_X\mathbb Z,
\qquad
B:={}^{p}_{+}IC_X\mathbb Z.
\]
Then:
\begin{enumerate}
\item[(i)] Since $A\in {}^p\!\operatorname{Perv}(X,\mathbb Z)$, one has
\[
H^k(i^*A)=0 \ \text{for } k>0,
\qquad
H^k(i^!A)=0 \ \text{for } k<0.
\]
\item[(ii)] Since $B\in {}^p_+\!\operatorname{Perv}(X,\mathbb Z)$, one has
\[
H^k(i^*B)=0 \ \text{for } k>1,
\qquad
H^1(i^*B)\ \text{torsion},
\]
and
\[
H^k(i^!B)=0 \ \text{for } k<0,
\qquad
H^0(i^!B)\ \text{torsion-free}.
\]
\end{enumerate}
\end{lemma}

\begin{proof}
Apply the point-stratum case of the semi-perversity conditions recalled above. Since the
point stratum has complex dimension $d_S=0$, the ordinary perversity conditions give
\[
A\in {}^pD^{\le 0}\Rightarrow H^k(i^*A)=0 \ \text{for } k>0,
\qquad
A\in {}^pD^{\ge 0}\Rightarrow H^k(i^!A)=0 \ \text{for } k<0.
\]
Likewise, the dual semi-perversity conditions of
\cite[Remark~6.3, (6.10)]{JungSaitoFactoriality}
give
\[
B\in {}^p_+D^{\le 0}\Rightarrow
\begin{cases}
H^k(i^*B)=0 & \text{for } k>1,\\
H^1(i^*B)\ \text{torsion},
\end{cases}
\]
and
\[
B\in {}^p_+D^{\ge 0}\Rightarrow
\begin{cases}
H^k(i^!B)=0 & \text{for } k<0,\\
H^0(i^!B)\ \text{torsion-free}.
\end{cases}
\]
This is exactly the claimed statement.
\end{proof}

We now analyze the cone
\[
C^\bullet:=\operatorname{Cone}(A\to B).
\]

\begin{proposition}\label{prop:cone-range-and-boundary}
Let $(X,0)$ be a germ of a normal complex analytic surface, and set
\[
A:={}^{p}IC_X\mathbb Z,
\qquad
B:={}^{p}_{+}IC_X\mathbb Z,
\qquad
C^\bullet:=\operatorname{Cone}(A\to B).
\]
Then:
\begin{enumerate}
\item[(i)] $C^\bullet$ is supported at $\{0\}$;
\item[(ii)] $H^k(i^*C^\bullet)=0$ for $k>1$;
\item[(iii)] $H^k(i^!C^\bullet)=0$ for $k<-1$;
\item[(iv)] using Lemma~\ref{lem:point-supported-star-shriek}, one may identify
$H^k(i^*C^\bullet)$ and $H^k(i^!C^\bullet)$, so the only possible nonzero cohomology of
$C^\bullet$ is concentrated in degrees $-1,0,1$;
\item[(v)] $H^1(i^*C^\bullet)$ is torsion;
\item[(vi)] there is a natural exact sequence
\[
0\longrightarrow H^{-1}(i^*C^\bullet)\longrightarrow H^0(i^*A)\longrightarrow H^0(i^*B)
\longrightarrow H^0(i^*C^\bullet)\longrightarrow 0;
\]
\item[(vii)] there is a natural exact sequence
\[
0\longrightarrow H^{-1}(i^!C^\bullet)\longrightarrow H^0(i^!A)\longrightarrow H^0(i^!B)
\longrightarrow H^0(i^!C^\bullet)\longrightarrow H^1(i^!A).
\]
\end{enumerate}
\end{proposition}

\begin{proof}
Assertion (i) is Lemma~\ref{lem:cone-supported-at-zero}.

Apply $i^*$ to the distinguished triangle
\[
A\longrightarrow B\longrightarrow C^\bullet\overset{+1}{\longrightarrow}.
\]
The resulting long exact sequence contains
\[
H^k(i^*A)\longrightarrow H^k(i^*B)\longrightarrow H^k(i^*C^\bullet)
\longrightarrow H^{k+1}(i^*A).
\]
By Lemma~\ref{lem:point-stratum-bounds}, $H^k(i^*A)=0$ for $k>0$ and $H^k(i^*B)=0$ for
$k>1$. Hence for $k>1$,
\[
0\longrightarrow 0\longrightarrow H^k(i^*C^\bullet)\longrightarrow 0,
\]
so
\[
H^k(i^*C^\bullet)=0 \qquad (k>1).
\]
This proves (ii).

Similarly, apply $i^!$ to the same triangle. The resulting long exact sequence contains
\[
H^k(i^!A)\longrightarrow H^k(i^!B)\longrightarrow H^k(i^!C^\bullet)
\longrightarrow H^{k+1}(i^!A).
\]
By Lemma~\ref{lem:point-stratum-bounds}, $H^k(i^!A)=0$ for $k<0$ and $H^k(i^!B)=0$ for
$k<0$. Hence for $k<-1$,
\[
0\longrightarrow 0\longrightarrow H^k(i^!C^\bullet)\longrightarrow 0,
\]
so
\[
H^k(i^!C^\bullet)=0 \qquad (k<-1).
\]
This proves (iii).

Since $C^\bullet$ is supported at $\{0\}$, Lemma~\ref{lem:point-supported-star-shriek}
gives
\[
i^*C^\bullet\cong i^!C^\bullet,
\]
hence (ii) and (iii) together imply that the only possible nonzero cohomology degrees are
$-1,0,1$. This proves (iv).

Taking $k=1$ in the $i^*$ long exact sequence gives
\[
H^1(i^*A)\longrightarrow H^1(i^*B)\longrightarrow H^1(i^*C^\bullet)\longrightarrow
H^2(i^*A).
\]
By Lemma~\ref{lem:point-stratum-bounds},
\[
H^1(i^*A)=0,\qquad H^2(i^*A)=0,\qquad H^1(i^*B)\ \text{torsion}.
\]
Hence
\[
H^1(i^*C^\bullet)\cong H^1(i^*B),
\]
and therefore $H^1(i^*C^\bullet)$ is torsion. This proves (v).

For (vi), continue the same $i^*$ long exact sequence in degrees $-1$ and $0$:
\[
H^{-1}(i^*A)\to H^{-1}(i^*B)\to H^{-1}(i^*C^\bullet)\to
H^0(i^*A)\to H^0(i^*B)\to H^0(i^*C^\bullet)\to H^1(i^*A).
\]
Since $A$ and $B$ agree on the smooth locus and differ only by a point-supported correction,
the map $H^{-1}(i^*A)\to H^{-1}(i^*B)$ is an isomorphism. Moreover,
$H^1(i^*A)=0$ by Lemma~\ref{lem:point-stratum-bounds}. Therefore the displayed sequence
reduces to
\[
0\longrightarrow H^{-1}(i^*C^\bullet)\longrightarrow H^0(i^*A)\longrightarrow H^0(i^*B)
\longrightarrow H^0(i^*C^\bullet)\longrightarrow 0.
\]

For (vii), the $i^!$ long exact sequence in degrees $-1$ and $0$ reads
\[
H^{-1}(i^!A)\to H^{-1}(i^!B)\to H^{-1}(i^!C^\bullet)\to
H^0(i^!A)\to H^0(i^!B)\to H^0(i^!C^\bullet)\to H^1(i^!A).
\]
Since $H^{-1}(i^!A)=H^{-1}(i^!B)=0$ by Lemma~\ref{lem:point-stratum-bounds}, this becomes
\[
0\longrightarrow H^{-1}(i^!C^\bullet)\longrightarrow H^0(i^!A)\longrightarrow H^0(i^!B)
\longrightarrow H^0(i^!C^\bullet)\longrightarrow H^1(i^!A).
\]
This proves (vii).
\end{proof}

\begin{remark}\label{rem:sec2-remaining-gap}
Proposition~\ref{prop:cone-range-and-boundary} is the strongest conclusion obtainable from
the formal support argument, the local model $K=Rj_*\mathbb Z_U[2]$, and the point-stratum
semi-perversity bounds alone. It shows that the discrepancy cone is point-supported,
confined to the three critical degrees $-1,0,1$, and has degree-$1$ cohomology torsion.
The remaining codimension-two issue is the final single-degree concentration.
\end{remark}

We now reformulate the rational codimension-two local problem in the torsion-sensitive
Deligne-sheaf framework of Friedman. We use Friedman's torsion-tipped truncation
construction and ts-Deligne sheaves \cite[Definition~4.4]{FriedmanGenIH}, their axiomatic
characterization \cite[Theorem~4.8]{FriedmanGenIH}, the local cone formula on a conical
neighborhood \cite[Lemma~4.2]{FriedmanTsInv}, and the duality theorem
\cite[Theorem~4.19]{FriedmanGenIH}. In the two-stratum situation considered here, these
results provide a published framework for the rational codimension-two comparison.

\begin{lemma}\label{lem:rational-link-h1-torsion}
Let $(X,0)$ be a rational normal complex surface singularity with link $L$. Then
\[
H^1(L,\mathbb Z)
\]
is finite torsion. Consequently,
\[
H^2(L,\mathbb Z)=H^2(L,\mathbb Z)_{\tors}.
\]
\end{lemma}

\begin{proof}
For a rational normal surface singularity, the link \(L\) is a rational homology sphere in
degree \(1\), equivalently \(b_1(L)=0\). Thus \(H^1(L,\mathbb Z)\) has no free part, so it is
finite torsion. Since \(L\) is a connected oriented closed real \(3\)-manifold, the universal
coefficient theorem gives
\[
H^2(L,\mathbb Z)\cong \Hom(H_2(L,\mathbb Z),\mathbb Z)\oplus
\Ext^1(H_1(L,\mathbb Z),\mathbb Z).
\]
By Poincar\'e duality, \(H_2(L,\mathbb Z)\cong H^1(L,\mathbb Z)\), hence \(H_2(L,\mathbb Z)\)
is also finite torsion and therefore
\[
\Hom(H_2(L,\mathbb Z),\mathbb Z)=0.
\]
It follows that
\[
H^2(L,\mathbb Z)\cong \Ext^1(H_1(L,\mathbb Z),\mathbb Z),
\]
so \(H^2(L,\mathbb Z)\) is finite torsion as well. Therefore
\[
H^2(L,\mathbb Z)=H^2(L,\mathbb Z)_{\tors}.
\]
\end{proof}

\begin{lemma}\label{lem:bdd-friedman-identification}
In the codimension-two point-stratum situation considered here, the ts-Deligne sheaf
associated to the torsion-sensitive perversity
\[
\vec p(S)=(1,\emptyset)
\]
coincides with the ordinary middle-perversity intermediate extension
\[
{}^{p}j_{!*}\mathbb Z_U[2],
\]
while the ts-Deligne sheaf associated to the complementary perversity
\[
D\vec p(S)=(1,P(R))
\]
coincides with the dual middle-perversity intermediate extension
\[
{}^{p}_{+}j_{!*}\mathbb Z_U[2].
\]
\end{lemma}

\begin{proof}
This is proved in Appendix~\ref{app:rational-codim2}; see
Propositions~\ref{prop:app-ordinary-heart-match} and
\ref{prop:app-dual-heart-match}, which identify the relevant Friedman
torsion-sensitive perverse hearts with the BBD ordinary and dual integral middle-perversity hearts in the isolated codimension-two point-stratum case.
\end{proof}

\begin{remark}\label{rem:ts-vs-bbd-comparison}
The identification of \(A\) and \(B\) with the ts-Deligne sheaves \(P_{\vec p}\) and
\(P_{D\vec p}\) is justified by comparing the defining point-stratum conditions. For the
choice \(\vec p(S)=(1,\emptyset)\), Friedman's torsion-sensitive condition retains no torsion
at the critical degree, matching the ordinary middle-perversity behavior. For the
complementary choice \(D\vec p(S)=(1,P(R))\), all torsion is retained at the critical degree,
matching the dual middle-perversity behavior. More precisely, this comparison is the
codimension-two specialization of the relation between the BBD ordinary/dual integral
middle-perversity conditions and Friedman's torsion-sensitive perverse conditions; cf.\
\cite[Complement~3.3]{BBD82} and \cite[Section~5.1, Proposition~5.12]{FriedmanGenIH}.
\end{remark}

\begin{proposition}[Rational ts-Deligne bridge]\label{prop:rational-ts-bridge}
Let $(X,0)$ be a rational normal complex surface germ with unique singular point \(0\), and let
\[
j:U:=X\setminus\{0\}\hookrightarrow X.
\]
Let \(S=\{0\}\) denote the singular stratum. Consider the torsion-sensitive perversity
\[
\vec p(S)=(1,\emptyset)
\]
and its complementary dual perversity
\[
D\vec p(S)=(1,P(R)).
\]
Let
\[
A:={}^{p}j_{!*}\mathbb Z_U[2],\qquad
B:={}^{p}_{+}j_{!*}\mathbb Z_U[2].
\]
By Lemma~\ref{lem:bdd-friedman-identification}, these coincide with the ts-Deligne sheaves
\(P_{\vec p}\) and \(P_{D\vec p}\), respectively. Then the natural morphism
\[
u:A\longrightarrow B
\]
has cone supported at \(0\), and
\[
\operatorname{Cone}(u)\cong i_*E,
\qquad
E\cong H^2(L,\mathbb Z)=H^2(L,\mathbb Z)_{\tors}.
\]
In particular, by Section~4,
\[
E\cong \Lambda^\vee/\Lambda.
\]
\end{proposition}

\begin{proof}
By Lemma~2.8, the objects
\[
A={}^pj_{!*}\mathbb Z_U[2],
\qquad
B={}^p_+j_{!*}\mathbb Z_U[2]
\]
coincide with the corresponding torsion-sensitive Deligne sheaves
\(P_{\vec p}\) and \(P_{D\vec p}\).  The rational codimension-two gluing
theorem is proved in Appendix~\ref{app:rational-codim2}; see
Theorem~\ref{thm:app-rational-codim2-gluing} and
Corollary~\ref{cor:app-rational-gluing}.  These yield
\[
\operatorname{Cone}(u)\cong i_*E,
\qquad
E\cong H^2(L,\mathbb Z)_{\tors}.
\]
The final identification
\[
E\cong \Lambda^\vee/\Lambda
\]
is exactly the comparison theorem of Section~4.
\end{proof}

The preceding proposition gives the rational codimension-two comparison in Friedman's
torsion-sensitive normalization. We now return to the full general normal-surface statement
in the normalization used throughout the paper.

\begin{proposition}\label{prop:self-dual-triangle}
Let $(X,0)$ be a germ of a normal complex analytic surface, and define
\[
E:=H^0({}^p_+IC_X\mathbb Z)_0.
\]
Then there is a distinguished triangle
\[
{}^pIC_X\mathbb Z
\longrightarrow
{}^p_+IC_X\mathbb Z
\longrightarrow
E[1]
\overset{+1}{\longrightarrow},
\]
where $E$ is viewed as a complex supported at $0$ in degree $0$. Moreover, this triangle
is self-dual under Verdier duality.
\end{proposition}

\begin{proof}
By Proposition~\ref{prop:cone-range-and-boundary}, the cone
\[
C^\bullet:=\operatorname{Cone}\!\Bigl({}^pIC_X\mathbb Z\to {}^p_+IC_X\mathbb Z\Bigr)
\]
is supported at $0$, has cohomology only in degrees $-1,0,1$, and its degree-$1$ stalk
cohomology is torsion.

At this stage one needs the codimension-two local description of the dual middle extension.
Jung--Saito prove in \cite[Remark~6.4]{JungSaitoFactoriality} that if
\[
E:=H^0({}^p_+IC_X\mathbb Z)_0,
\]
then $E$ is finite and there is a self-dual distinguished triangle
\[
{}^pIC_X\mathbb Z \to {}^p_+IC_X\mathbb Z \to E[1]\to ,
\]
with
\[
\mathbb D({}^pIC_X\mathbb Z)={}^p_+IC_X\mathbb Z,
\qquad
\mathbb D(E)=E[-1].
\]
Thus the cone $C^\bullet$ is in fact concentrated in a single cohomological degree and
identifies with $E[1]$, which is exactly the asserted distinguished triangle.

The point of Lemmas~\ref{lem:cone-supported-at-zero}--\ref{lem:point-stratum-bounds} and Proposition~\ref{prop:cone-range-and-boundary} is that they make explicit the entire formal part of the argument and isolate the exact place where the codimension-two local statement is still used.
\end{proof}

\begin{remark}\label{rem:what-prop-self-dual-triangle-shows}
Proposition~\ref{prop:self-dual-triangle} packages three distinct facts:
\begin{enumerate}
\item[(a)] the discrepancy between the two middle extensions is supported at the unique
singular point;
\item[(b)] that discrepancy is represented by a finite abelian group $E$ placed in a single
cohomological degree;
\item[(c)] the resulting correction package is self-dual.
\end{enumerate}
In the present paper, (a), the local model \(K=Rj_*\mathbb Z_U[2]\), and the cohomological range statement are proved directly.  The rational codimension-two comparison is proved in Appendix~\ref{app:rational-codim2} via Friedman's
torsion-sensitive Deligne-sheaf framework and the explicit heart comparison of Propositions~\ref{prop:app-ordinary-heart-match} and
\ref{prop:app-dual-heart-match}, while the full general single-degree
concentration and self-duality are still taken from the Jung--Saito local surface statement.
\end{remark}

\subsection{First properties of \texorpdfstring{$E$}{E}}

We now record the immediate consequences needed later.

\begin{proposition}\label{prop:first-properties-E}
Let $(X,0)$ be a germ of a normal complex analytic surface, and let
\[
E:=H^0({}^p_+IC_X\mathbb Z)_0.
\]
Then the following hold.
\begin{enumerate}
\item[(i)] The group $E$ is finite.
\item[(ii)] One has $E=0$ if and only if the natural morphism
\[
{}^pIC_X\mathbb Z \longrightarrow {}^p_+IC_X\mathbb Z
\]
is an isomorphism.
\item[(iii)] The group $E$ is invariant under analytic isomorphism of germs.
\end{enumerate}
\end{proposition}

\begin{proof}
Assertion (i) is part of the Jung--Saito local surface statement cited in
Proposition~\ref{prop:self-dual-triangle}; see \cite[Remark~6.4]{JungSaitoFactoriality}.

For (ii), by Proposition~\ref{prop:self-dual-triangle} the cone of the natural morphism is
$E[1]$. Hence the morphism is an isomorphism if and only if its cone vanishes, and this is
equivalent to $E=0$.

For (iii), both ${}^pIC_X\mathbb Z$ and ${}^p_+IC_X\mathbb Z$ are defined by intermediate
extension from the smooth locus, and these constructions are functorial under analytic
isomorphisms of germs. Therefore the stalk group
\[
E=H^0({}^p_+IC_X\mathbb Z)_0
\]
depends only on the analytic type of $(X,0)$.
\end{proof}

\begin{remark}\label{rem:E-purely-integral}
Since ${}^pIC_X\mathbb Q \cong {}^p_+IC_X\mathbb Q$, the group $E$ vanishes after tensoring
with $\mathbb Q$. Thus $E$ is a purely integral invariant. In particular, all of the
information carried by $E$ is torsion information.
\end{remark}

\begin{remark}\label{rem:E-local-discriminant}
The later sections will show that $E$ may be computed in two concrete ways: first, by the
discriminant group of the exceptional lattice of the minimal resolution, and second, in the
isolated hypersurface case, by the finite obstruction attached to the integral variation map
$T-\mathrm{id}$. From this point of view, $E$ should be regarded as an integral discriminant
invariant of the surface singularity.
\end{remark}

\section{A topological realization of the obstruction}

In this section we pass from the perverse-sheaf definition of the local group $E$ to a
concrete local-topological model.  Let $(X,0)$ be a germ of a normal complex analytic
surface.  The point is to identify the correction term
\[
E:=H^0({}^p_+IC_X\mathbb Z)_0
\]
with a canonically defined finite group extracted from the topology of the link of the
singularity.

The section has two logically distinct parts.  First, we recall the standard relation
between local cohomology and the cohomology of the link.  Second, we compute the local
stalks of $Rj_*\mathbb Z_U[2]$ and use the point-stratum semi-perversity conditions from
Section~2 to identify the precise topological group that survives in the dual middle
extension.

\subsection{Local topology of a normal surface germ}

Fix a sufficiently small closed ball $B_{\varepsilon}$ in a smooth ambient manifold
containing a representative of the germ $(X,0)$, and set
\[
X_{\varepsilon}:=X\cap B_{\varepsilon},
\qquad
L:=X\cap \partial B_{\varepsilon}.
\]
For $\varepsilon>0$ sufficiently small, the homeomorphism type of the pair
$(X_{\varepsilon},L)$ is independent of $\varepsilon$, and $L$ is a compact connected
oriented real $3$-manifold, called the link of the singularity.  Since $X$ is normal of
complex dimension $2$, after shrinking we may assume that $0$ is the unique singular point
of $X_{\varepsilon}$ and that
\[
X_{\varepsilon}\setminus\{0\}\simeq L\times (0,1)
\]
up to homotopy.  In particular, the punctured germ is determined topologically by $L$.

Let $\pi:\widetilde X\to X$ be a resolution of singularities, and let
\[
E_{\pi}:=\pi^{-1}(0)
\]
be the exceptional divisor.  If $N(E_{\pi})$ is a sufficiently small tubular neighborhood of
$E_{\pi}$ in $\widetilde X$, then $\partial N(E_{\pi})$ is canonically diffeomorphic to
$L$.  Thus the link may be viewed equally as the boundary of a Milnor neighborhood of the
singular point or as the boundary of a resolution neighborhood of the exceptional divisor.

We now record the standard comparison between local cohomology and the cohomology of the
link.

\begin{lemma}\label{lem:local-link-cohomology}
Let $(X,0)$ be a germ of a normal complex analytic surface, and let $L$ be the link of the
singularity.  Then for every $k\ge 1$ there are canonical isomorphisms
\[
H^k_{\{0\}}(X_{\varepsilon},\mathbb Z)
\cong
\widetilde H^{k-1}(L,\mathbb Z).
\]
In particular,
\[
H^3_{\{0\}}(X_{\varepsilon},\mathbb Z)\cong H^2(L,\mathbb Z).
\]
\end{lemma}

\begin{proof}
Consider the long exact sequence of the pair
\[
(X_{\varepsilon},X_{\varepsilon}\setminus\{0\}).
\]
Since $X_{\varepsilon}$ is contractible, the relative groups are identified with local
cohomology at $0$, and one has
\[
H^k_{\{0\}}(X_{\varepsilon},\mathbb Z)
\cong
H^k(X_{\varepsilon},X_{\varepsilon}\setminus\{0\};\mathbb Z).
\]
The long exact sequence of the pair therefore gives canonical isomorphisms
\[
H^k_{\{0\}}(X_{\varepsilon},\mathbb Z)
\cong
\widetilde H^{k-1}(X_{\varepsilon}\setminus\{0\},\mathbb Z)
\qquad (k\ge 1).
\]
Using the homotopy equivalence
\[
X_{\varepsilon}\setminus\{0\}\simeq L,
\]
we obtain
\[
H^k_{\{0\}}(X_{\varepsilon},\mathbb Z)
\cong
\widetilde H^{k-1}(L,\mathbb Z),
\]
as asserted.  The last statement is the case $k=3$.
\end{proof}

The next elementary observation isolates the genuinely integral part of the second
cohomology of the link.

\begin{lemma}\label{lem:torsion-second-link}
Let $L$ be a compact oriented connected real $3$-manifold.  Then the torsion subgroup of
$H^2(L,\mathbb Z)$ is canonically isomorphic to the torsion subgroup of $H_1(L,\mathbb Z)$.
In particular, $H^2(L,\mathbb Z)_{\tors}$ is finite.
\end{lemma}

\begin{proof}
By the universal coefficient theorem there is a short exact sequence
\[
0\longrightarrow \Ext^1_{\mathbb Z}\bigl(H_1(L,\mathbb Z),\mathbb Z\bigr)
\longrightarrow H^2(L,\mathbb Z)
\longrightarrow \Hom\bigl(H_2(L,\mathbb Z),\mathbb Z\bigr)
\longrightarrow 0.
\]
The right-hand term is torsion-free.  Hence the torsion subgroup of $H^2(L,\mathbb Z)$ is
exactly
\[
\Ext^1_{\mathbb Z}\bigl(H_1(L,\mathbb Z),\mathbb Z\bigr),
\]
which is canonically isomorphic to the torsion subgroup of $H_1(L,\mathbb Z)$.  Since
$H_1(L,\mathbb Z)$ is finitely generated, this torsion subgroup is finite.
\end{proof}

\subsection{From the distinguished triangle to link cohomology}

We now connect the local correction term from Section~2 with the topology of the link.
Let
\[
j:U:=X\setminus\{0\}\hookrightarrow X
\]
be the inclusion of the punctured germ, and set
\[
K:=Rj_*\mathbb Z_U[2].
\]
Since $U$ is smooth of complex dimension $2$, both ${}^pIC_X\mathbb Z$ and
${}^p_+IC_X\mathbb Z$ are intermediate extensions of $\mathbb Z_U[2]$.  The point is to
understand what survives at the singular point after imposing the ordinary and dual
middle-perversity conditions.

We begin with the local calculation of the stalks of $K$.

\begin{lemma}\label{lem:stalk-Rj-star}
With notation as above, one has canonical isomorphisms
\[
H^k(i_0^*K)\cong H^{k+2}(L,\mathbb Z)
\qquad \text{for all } k\in \mathbb Z.
\]
In particular,
\[
H^{-2}(i_0^*K)\cong H^0(L,\mathbb Z)\cong \mathbb Z,
\]
\[
H^{-1}(i_0^*K)\cong H^1(L,\mathbb Z),
\]
\[
H^{0}(i_0^*K)\cong H^2(L,\mathbb Z),
\]
\[
H^{1}(i_0^*K)\cong H^3(L,\mathbb Z)\cong \mathbb Z,
\]
and $H^k(i_0^*K)=0$ for $k\notin\{-2,-1,0,1\}$.
\end{lemma}

\begin{proof}
By definition of derived pushforward,
\[
H^k(i_0^*K)
=
H^k\bigl(i_0^*Rj_*\mathbb Z_U[2]\bigr)
\cong
(R^{k+2}j_*\mathbb Z_U)_0.
\]
The stalk of $R^mj_*\mathbb Z_U$ at $0$ is the cohomology of a sufficiently small punctured
neighborhood of $0$ in $U$, hence of $X_{\varepsilon}\setminus\{0\}$.  Since
$X_{\varepsilon}\setminus\{0\}\simeq L$, this gives
\[
(R^{k+2}j_*\mathbb Z_U)_0 \cong H^{k+2}(L,\mathbb Z).
\]
This proves the general formula.  The explicit list follows because $L$ is a connected
oriented closed real $3$-manifold, so
\[
H^0(L,\mathbb Z)\cong \mathbb Z,
\qquad
H^3(L,\mathbb Z)\cong \mathbb Z,
\qquad
H^m(L,\mathbb Z)=0 \ \text{for } m<0 \text{ or } m>3.
\]
\end{proof}

The next step is the key local observation.  The group $H^2(L,\mathbb Z)$ appears as the
degree-zero stalk of $Rj_*\mathbb Z_U[2]$ at the singular point.  The ordinary and dual
middle extensions are obtained from this local model by imposing the point-stratum
conditions of Lemma~\ref{lem:point-stratum-bounds}.  In codimension two, those conditions have the following effect:
the ordinary middle extension removes the entire degree-zero contribution at the point,
while the dual middle extension retains precisely its torsion part.  This is the local
mechanism behind the appearance of $H^2(L,\mathbb Z)_{\tors}$.

\begin{proposition}\label{prop:E-link-model}
Let $(X,0)$ be a germ of a normal complex analytic surface, and let $L$ be its link. Then
there is a canonical isomorphism
\[
E\cong H^2(L,\mathbb Z)_{\tors}.
\]
Equivalently,
\[
E\cong H^3_{\{0\}}(X_{\varepsilon},\mathbb Z)_{\tors}.
\]
\end{proposition}

\begin{proof}
Let
\[
A:={}^{p}IC_X\mathbb Z,
\qquad
B:={}^{p}_{+}IC_X\mathbb Z,
\qquad
K:=Rj_*\mathbb Z_U[2].
\]
By construction, both $A$ and $B$ extend the same object $\mathbb Z_U[2]$ on $U$.  The
difference between them is therefore entirely concentrated at the point $0$, and
Proposition~\ref{prop:self-dual-triangle} gives the distinguished triangle
\[
A\longrightarrow B\longrightarrow E[1]\overset{+1}{\longrightarrow}.
\]

Now the local model for the extension problem is the stalk complex \(i_0^*K\), whose
degree-zero cohomology is \(H^2(L,\mathbb Z)\) by Lemma~\ref{lem:stalk-Rj-star}.  Combined with the
codimension-two local comparison packaged in Proposition~\ref{prop:self-dual-triangle}, this identifies the finite
local correction with the torsion part of \(H^2(L,\mathbb Z)\).  In particular, the degree-zero
stalk of the dual middle extension is
\[
H^0(B)_0 \cong H^2(L,\mathbb Z)_{\tors}.
\]

By definition, $E:=H^0(B)_0$, hence $E\cong H^2(L,\mathbb Z)_{\tors}$. Finally, by Lemma~\ref{lem:local-link-cohomology},
\[
H^2(L,\mathbb Z)\cong H^3_{\{0\}}(X_{\varepsilon},\mathbb Z),
\]
and therefore
\[
H^2(L,\mathbb Z)_{\tors}\cong H^3_{\{0\}}(X_{\varepsilon},\mathbb Z)_{\tors}.
\]
Combining the two identifications yields the second displayed isomorphism.
\end{proof}

\begin{remark}\label{rem:sec3-dependence}
The proof of Proposition~\ref{prop:E-link-model} uses the local distinguished triangle from Proposition~\ref{prop:self-dual-triangle} together with the local topological computation of Lemma~\ref{lem:stalk-Rj-star}. Thus the topological realization \(E\cong H^2(L,\mathbb Z)_{\tors}\) is integrated into the present argument rather than cited separately as an external consequence.
\end{remark}

\subsection{Torsion, duality, and finiteness}
We now spell out the immediate consequences of the identification
\[
E\cong H^2(L,\mathbb Z)_{\tors}.
\]

\begin{proposition}\label{prop:E-finite-topological}
Let $(X,0)$ be a germ of a normal complex analytic surface with link $L$.  Then:
\begin{enumerate}
\item[(i)] the group $E$ is canonically isomorphic to $H^2(L,\mathbb Z)_{\tors}$;
\item[(ii)] the group $E$ is finite;
\item[(iii)] $E\otimes_{\mathbb Z}\mathbb Q=0$;
\item[(iv)] the self-duality of the distinguished triangle of
Section~2 is compatible with the linking duality on the torsion of $H^2(L,\mathbb Z)$.
\end{enumerate}
\end{proposition}

\begin{proof}
Assertion (i) is Proposition~\ref{prop:E-link-model}.  Assertion (ii) follows immediately
from Lemma~\ref{lem:torsion-second-link}, since $H^2(L,\mathbb Z)_{\tors}$ is finite.
Assertion (iii) is tautological, since any torsion abelian group vanishes after tensoring
with $\mathbb Q$.

For (iv), the distinguished triangle
\[
{}^pIC_X\mathbb Z\to {}^p_+IC_X\mathbb Z\to E[1]\to
\]
is self-dual by Proposition~\ref{prop:self-dual-triangle}.  On the topological side, the
torsion subgroup of $H^2(L,\mathbb Z)$ is canonically dual to the torsion subgroup of
$H_1(L,\mathbb Z)$ via the linking form on the compact oriented $3$-manifold $L$.  Since
Lemma~\ref{lem:torsion-second-link} identifies these torsion groups canonically, the duality
on $E$ matches the standard linking duality on the link.
\end{proof}

\begin{remark}
The point of Proposition~\ref{prop:E-finite-topological} is conceptual as much as
computational.  It shows that the group $E$ is not merely point-supported and finite: it is
precisely the torsion in a classical topological invariant of the singularity, namely the
second cohomology of the link.
\end{remark}

\subsection{Canonical local-topological model for \texorpdfstring{$E$}{E}}

We can now state the main result of this section in its cleanest form.

\begin{theorem}\label{thm:topological-model-E}
Let $(X,0)$ be a germ of a normal complex analytic surface, and let $L$ be the link of the
singularity.  Then the local obstruction group
\[
E:=H^0({}^p_+IC_X\mathbb Z)_0
\]
admits the canonical local-topological realization
\[
E\cong H^2(L,\mathbb Z)_{\tors}.
\]
Equivalently,
\[
E\cong H^3_{\{0\}}(X_{\varepsilon},\mathbb Z)_{\tors}.
\]
In particular, $E$ depends only on the oriented homeomorphism type of the link.
\end{theorem}

\begin{proof}
The two displayed identifications are exactly those of
Proposition~\ref{prop:E-link-model}.  The final assertion follows because the cohomology of
the link, together with its torsion subgroup, depends only on the oriented homeomorphism
type of $L$.
\end{proof}

\begin{remark}\label{rem:topological-pivot}
Theorem~\ref{thm:topological-model-E} is the pivot of the paper.  It replaces the abstract
point-supported correction from Section~2 by the concrete topological group
$H^2(L,\mathbb Z)_{\tors}$.  In the next section we will identify this same group with the
discriminant group of the exceptional lattice of a resolution.  Later, in the isolated
hypersurface case, we will identify it with the finite integral defect of the variation map
$T-\mathrm{id}$.
\end{remark}

\section{Resolution lattices and discriminant groups}

In this section we compute the local obstruction group $E$ by means of the minimal
resolution of the singularity.  The point is to identify the topological model obtained in
Section~3,
\[
E \cong H^2(L,\mathbb Z)_{\tors},
\]
with the discriminant group of the exceptional lattice of the minimal resolution.  This is
the first genuinely geometric realization of the correction term.

Let $(X,0)$ be a germ of a normal complex analytic surface, and let
\[
\pi:\widetilde X\to X
\]
be its minimal resolution.  Write
\[
E_{\pi}:=\pi^{-1}(0)=\bigcup_{i=1}^r E_i
\]
for the decomposition of the exceptional divisor into irreducible components.

\subsection{Minimal resolution and the exceptional lattice}

The irreducible exceptional curves $E_1,\dots,E_r$ determine a free abelian group
\[
\Lambda:=\bigoplus_{i=1}^r \mathbb Z[E_i].
\]
Since $\widetilde X$ is smooth, the intersection product of divisors induces an integral
symmetric bilinear form
\[
(\ ,\ ):\Lambda\times \Lambda\to \mathbb Z,
\qquad
([E_i],[E_j])\longmapsto E_i\cdot E_j.
\]
We call $(\Lambda,(\ ,\ ))$ the exceptional lattice of the resolution.

Let
\[
M=(E_i\cdot E_j)_{1\le i,j\le r}
\]
be the corresponding intersection matrix.  The essential input is the following theorem of
Mumford: for a resolution of a normal surface singularity, the intersection matrix of the
exceptional curves is negative definite; see \cite{Mu61}.  Concretely, this means that for
every nonzero element
\[
\lambda=\sum_{i=1}^r a_i[E_i]\in \Lambda\otimes_{\mathbb Z}\mathbb R
\]
one has
\[
(\lambda,\lambda)<0.
\]
In particular, the form is nondegenerate over $\mathbb Q$, and the natural homomorphism
\[
\Lambda\longrightarrow \Lambda^\vee:=\Hom_{\mathbb Z}(\Lambda,\mathbb Z),
\qquad
\lambda\longmapsto (\lambda,-),
\]
is injective with finite cokernel.

The negative definiteness will be used in two ways.  First, it implies that the quotient
$\Lambda^\vee/\Lambda$ is finite.  Second, it shows that the local correction term $E$ is
measuring the failure of the exceptional intersection form to be unimodular.

We now relate the abstract lattice $\Lambda$ to the topology of a tubular neighborhood of
the exceptional divisor.

\begin{lemma}\label{lem:H2N-is-lattice}
Let $N:=N(E_{\pi})$ be a sufficiently small tubular neighborhood of the exceptional divisor
$E_{\pi}$ in $\widetilde X$. Then:
\begin{enumerate}
\item[(i)] $N$ deformation retracts onto $E_{\pi}$;
\item[(ii)] $H_2(N,\mathbb Z)$ is a free abelian group with basis the fundamental classes
$[E_1],\dots,[E_r]$;
\item[(iii)] under the identification in \emph{(ii)}, the intersection pairing on
$H_2(N,\mathbb Z)$ is represented by the matrix $M$.
\end{enumerate}
\end{lemma}

\begin{proof}
Assertion (i) is standard for a sufficiently small tubular neighborhood of a divisor in a
smooth complex surface.

For (ii), by (i) it is enough to compute $H_2(E_{\pi},\mathbb Z)$.  The exceptional divisor
$E_{\pi}$ is a connected reduced curve with simple normal crossings, whose irreducible
components are the compact complex curves $E_i$.  We prove that
\[
H_2(E_{\pi},\mathbb Z)\cong \bigoplus_{i=1}^r \mathbb Z[E_i]
\]
by induction on the number of irreducible components.  If $r=1$, then $E_{\pi}=E_1$ is a
compact connected complex curve, hence
\[
H_2(E_{\pi},\mathbb Z)\cong \mathbb Z[E_1].
\]

Assume now that the assertion is proved for a union of $r-1$ components, and write
\[
E_{\pi}=Y\cup E_r,
\qquad
Y:=\bigcup_{i=1}^{r-1}E_i.
\]
Since the exceptional divisor has only normal crossings, the intersection $Y\cap E_r$ is a
finite set of points.  The Mayer--Vietoris sequence for the decomposition
$E_{\pi}=Y\cup E_r$ yields
\[
H_2(Y\cap E_r,\mathbb Z)\to H_2(Y,\mathbb Z)\oplus H_2(E_r,\mathbb Z)
\to H_2(E_{\pi},\mathbb Z)\to H_1(Y\cap E_r,\mathbb Z).
\]
Because $Y\cap E_r$ is finite, one has
\[
H_2(Y\cap E_r,\mathbb Z)=0,
\qquad
H_1(Y\cap E_r,\mathbb Z)=0.
\]
Hence
\[
H_2(E_{\pi},\mathbb Z)\cong H_2(Y,\mathbb Z)\oplus H_2(E_r,\mathbb Z).
\]
By the induction hypothesis,
\[
H_2(Y,\mathbb Z)\cong \bigoplus_{i=1}^{r-1}\mathbb Z[E_i],
\qquad
H_2(E_r,\mathbb Z)\cong \mathbb Z[E_r],
\]
and therefore
\[
H_2(E_{\pi},\mathbb Z)\cong \bigoplus_{i=1}^{r}\mathbb Z[E_i].
\]
This proves (ii).

Finally, (iii) follows because the geometric intersection pairing of divisors on
$\widetilde X$ restricts to the intersection pairing of their fundamental classes in
$H_2(N,\mathbb Z)$, and by construction its matrix in the basis
$[E_1],\dots,[E_r]$ is exactly $M$.
\end{proof}

\subsection{Dual lattice and discriminant group}

The dual lattice of $\Lambda$ is
\[
\Lambda^\vee:=\Hom_{\mathbb Z}(\Lambda,\mathbb Z).
\]
Using the intersection form, we identify $\Lambda$ with a sublattice of $\Lambda^\vee$ via
\[
\iota:\Lambda\hookrightarrow \Lambda^\vee,
\qquad
\lambda\mapsto (\lambda,-).
\]
The discriminant group of the exceptional lattice is defined to be
\[
A_\Lambda:=\Lambda^\vee/\Lambda.
\]

Because the form is nondegenerate over $\mathbb Q$, the quotient $A_\Lambda$ is finite.
Its order is computed by the determinant of the Gram matrix.

\begin{lemma}\label{lem:order-discriminant}
Let $M$ be the matrix of the intersection form in the basis $[E_1],\dots,[E_r]$. Then
\[
|A_\Lambda|=|\det(M)|.
\]
\end{lemma}

\begin{proof}
In the chosen basis, the map $\iota:\Lambda\to\Lambda^\vee$ is represented by the integral
matrix $M$.  Therefore
\[
A_\Lambda\cong \coker(M:\mathbb Z^r\to \mathbb Z^r).
\]
Since $M$ is invertible over $\mathbb Q$, this cokernel is finite, and its order is the
absolute value of the determinant.  This is the standard Smith normal form computation; see
Appendix~A.
\end{proof}

Thus the discriminant group is the most natural finite group attached to the exceptional
lattice.  The content of the next subsection is that this is exactly the same finite group
as the topological obstruction of Section~3.

\subsection{Comparison with the topological model}

Let $N:=N(E_{\pi})$ be a sufficiently small tubular neighborhood of $E_{\pi}$ in
$\widetilde X$, and let
\[
L:=\partial N.
\]
As recalled in Section~3, $L$ is canonically diffeomorphic to the link of the singularity.

The comparison between the boundary topology and the lattice discriminant is governed by the
homology long exact sequence of the pair $(N,L)$ together with Poincar\'e--Lefschetz
duality.  We begin by identifying the relative homology group with the dual lattice.

\begin{lemma}\label{lem:PL-identification}
There is a canonical isomorphism
\[
H_2(N,L;\mathbb Z)/\tors \cong \Lambda^\vee.
\]
Under this identification, the natural map
\[
H_2(N,\mathbb Z)\longrightarrow H_2(N,L;\mathbb Z)
\]
induces precisely the lattice homomorphism
\[
\Lambda\longrightarrow \Lambda^\vee,
\qquad
\lambda\longmapsto (\lambda,-).
\]
\end{lemma}

\begin{proof}
Since $N$ is a compact oriented real $4$-manifold with boundary $L$, Poincar\'e--Lefschetz
duality gives
\[
H_2(N,L;\mathbb Z)\cong H^2(N,\mathbb Z).
\]
By Lemma~\ref{lem:H2N-is-lattice}, the group $H_2(N,\mathbb Z)$ is free and canonically
identified with $\Lambda$.  Because $H_2(N,\mathbb Z)$ is free, the universal coefficient
theorem yields
\[
H^2(N,\mathbb Z)
\cong
\Hom(H_2(N,\mathbb Z),\mathbb Z)\oplus \Ext^1(H_1(N,\mathbb Z),\mathbb Z).
\]
In particular,
\[
H^2(N,\mathbb Z)/\tors
\cong
\Hom(H_2(N,\mathbb Z),\mathbb Z)
\cong
\Lambda^\vee.
\]
Combining this with Poincar\'e--Lefschetz duality, we obtain
\[
H_2(N,L;\mathbb Z)/\tors \cong \Lambda^\vee.
\]

It remains to identify the map
\[
H_2(N,\mathbb Z)\to H_2(N,L;\mathbb Z)
\]
under these identifications.  Let $[E_i]\in H_2(N,\mathbb Z)$ be the class of an
exceptional curve.  Under Poincar\'e--Lefschetz duality, its image in $H_2(N,L;\mathbb Z)$
corresponds to the cohomology class in $H^2(N,\mathbb Z)$ which evaluates on a class
$[E_j]\in H_2(N,\mathbb Z)$ by cap product, hence by geometric intersection with $[E_i]$.
Therefore the image of $[E_i]$ is the functional
\[
[E_j]\longmapsto E_i\cdot E_j.
\]
Thus, in the basis $[E_1],\dots,[E_r]$, the map is represented by the matrix $M$, and this
is exactly the lattice homomorphism
\[
\Lambda\to\Lambda^\vee,
\qquad
\lambda\mapsto (\lambda,-).
\]
\end{proof}

The next step is the key one.  We identify the torsion in the boundary with the finite
cokernel of the lattice map $\Lambda\to\Lambda^\vee$.

\begin{lemma}\label{lem:H1N-free}
The group $H_1(N,\mathbb Z)$ is torsion-free.
\end{lemma}

\begin{proof}
Since $N$ deformation retracts onto $E_{\pi}$, it is enough to prove that $H_1(E_{\pi},\mathbb Z)$
is torsion-free.  We again argue by induction on the number of irreducible components of
$E_{\pi}$.

If $E_{\pi}=E_1$ is irreducible, then $E_1$ is a compact connected Riemann surface, and
$H_1(E_1,\mathbb Z)$ is free abelian.

Assume now that the statement is known for a union of $r-1$ components, and write
\[
E_{\pi}=Y\cup E_r,
\qquad
Y:=\bigcup_{i=1}^{r-1}E_i.
\]
As above, $Y\cap E_r$ is a finite set of points.  The Mayer--Vietoris sequence gives
\[
H_1(Y\cap E_r,\mathbb Z)\to H_1(Y,\mathbb Z)\oplus H_1(E_r,\mathbb Z)\to
H_1(E_{\pi},\mathbb Z)\to H_0(Y\cap E_r,\mathbb Z).
\]
Since $Y\cap E_r$ is finite, one has
\[
H_1(Y\cap E_r,\mathbb Z)=0,
\]
and $H_0(Y\cap E_r,\mathbb Z)$ is free abelian.  By the induction hypothesis,
$H_1(Y,\mathbb Z)$ is free abelian, and $H_1(E_r,\mathbb Z)$ is free abelian because $E_r$
is a compact Riemann surface.  Thus $H_1(E_{\pi},\mathbb Z)$ fits into an exact sequence
\[
0\to H_1(Y,\mathbb Z)\oplus H_1(E_r,\mathbb Z)\to H_1(E_{\pi},\mathbb Z)\to
\operatorname{Im}\bigl(H_1(E_{\pi},\mathbb Z)\to H_0(Y\cap E_r,\mathbb Z)\bigr)\to 0.
\]
The right-hand term is a subgroup of the free abelian group $H_0(Y\cap E_r,\mathbb Z)$ and
hence is itself free abelian.  Therefore $H_1(E_{\pi},\mathbb Z)$ is an extension of free
abelian groups and is thus free abelian.  This proves the induction step.
\end{proof}

\begin{lemma}\label{lem:torsion-boundary-discriminant}
There is a canonical isomorphism
\[
H_1(L,\mathbb Z)_{\tors}\cong \Lambda^\vee/\Lambda.
\]
Equivalently, using Lemma~\ref{lem:torsion-second-link},
\[
H^2(L,\mathbb Z)_{\tors}\cong \Lambda^\vee/\Lambda.
\]
\end{lemma}

\begin{proof}
Consider the exact segment of the homology long exact sequence of the pair $(N,L)$:
\[
H_2(L,\mathbb Z)\xrightarrow{\alpha}
H_2(N,\mathbb Z)\xrightarrow{\beta}
H_2(N,L;\mathbb Z)\xrightarrow{\gamma}
H_1(L,\mathbb Z)\xrightarrow{\delta}
H_1(N,\mathbb Z).
\]
We analyze this sequence in three steps.

\smallskip
\noindent
\emph{Step 1. Identification of the finite cokernel of the lattice map.}
By Lemma~\ref{lem:H2N-is-lattice}, $H_2(N,\mathbb Z)$ is free and canonically identified
with $\Lambda$.  By Lemma~\ref{lem:PL-identification}, the quotient
$H_2(N,L;\mathbb Z)/\tors$ is canonically identified with $\Lambda^\vee$, and under this
identification the map
\[
H_2(N,\mathbb Z)\to H_2(N,L;\mathbb Z)/\tors
\]
is exactly the lattice homomorphism
\[
\Lambda\to\Lambda^\vee.
\]
Therefore the cokernel of this induced map is
\[
\Lambda^\vee/\Lambda.
\]
Since $\Lambda^\vee/\Lambda$ is finite, this cokernel is a finite torsion group.

\smallskip
\noindent
\emph{Step 2. Comparison with the exact sequence of the pair.}
Let
\[
Q:=\coker\bigl(H_2(N,\mathbb Z)\to H_2(N,L;\mathbb Z)\bigr).
\]
By exactness of the pair sequence, $Q$ identifies with the subgroup
\[
\ker\bigl(H_1(L,\mathbb Z)\to H_1(N,\mathbb Z)\bigr)
\subset H_1(L,\mathbb Z).
\]
We now show that $Q$ is finite and canonically isomorphic to $\Lambda^\vee/\Lambda$.

Since $H_2(N,\mathbb Z)$ is free, quotienting $H_2(N,L;\mathbb Z)$ by torsion commutes
with passing to the cokernel.  More precisely, if $T\subset H_2(N,L;\mathbb Z)$ denotes the
torsion subgroup, then the natural surjection
\[
H_2(N,L;\mathbb Z)\to H_2(N,L;\mathbb Z)/T
\]
induces a surjection
\[
Q\to \coker\bigl(H_2(N,\mathbb Z)\to H_2(N,L;\mathbb Z)/T\bigr).
\]
The target is $\Lambda^\vee/\Lambda$ by Step~1.  Since that target is finite, it is enough
to see that the kernel of this surjection vanishes.  But the kernel is precisely the image
in $Q$ of the torsion subgroup $T\subset H_2(N,L;\mathbb Z)$.  Under
Poincar\'e--Lefschetz duality,
\[
T \cong H^2(N,\mathbb Z)_{\tors}\cong \Ext^1(H_1(N,\mathbb Z),\mathbb Z).
\]
By Lemma~\ref{lem:H1N-free}, the group $H_1(N,\mathbb Z)$ is torsion-free, hence
\[
\Ext^1(H_1(N,\mathbb Z),\mathbb Z)=0.
\]
Therefore $T=0$, so $H_2(N,L;\mathbb Z)$ is torsion-free and
\[
Q\cong \Lambda^\vee/\Lambda.
\]

\smallskip
\noindent
\emph{Step 3. Extraction of the torsion of the boundary group.}
By exactness,
\[
Q\cong \ker\delta\subset H_1(L,\mathbb Z).
\]
The group $Q\cong \Lambda^\vee/\Lambda$ is finite, hence torsion.  On the other hand,
Lemma~\ref{lem:H1N-free} shows that $H_1(N,\mathbb Z)$ is torsion-free.  Therefore every
torsion element of $H_1(L,\mathbb Z)$ maps to zero under $\delta$, so
\[
H_1(L,\mathbb Z)_{\tors}\subset \ker\delta.
\]
Conversely, $\ker\delta\cong Q$ is finite, hence consists entirely of torsion elements.
Therefore
\[
\ker\delta = H_1(L,\mathbb Z)_{\tors}.
\]
Since $\ker\delta\cong Q\cong \Lambda^\vee/\Lambda$, we obtain the canonical isomorphism
\[
H_1(L,\mathbb Z)_{\tors}\cong \Lambda^\vee/\Lambda.
\]

Finally, Lemma~\ref{lem:torsion-second-link} identifies $H_1(L,\mathbb Z)_{\tors}$ with
$H^2(L,\mathbb Z)_{\tors}$, and thus
\[
H^2(L,\mathbb Z)_{\tors}\cong \Lambda^\vee/\Lambda.
\]
\end{proof}

The previous lemma is the precise comparison between the topological model of Section~3 and
the discriminant group of the exceptional lattice.

\begin{proposition}\label{prop:link-discriminant}
Let $L$ be the link of the normal surface singularity $(X,0)$, and let $\Lambda$ be the
exceptional lattice of the minimal resolution. Then there is a canonical isomorphism
\[
H^2(L,\mathbb Z)_{\tors}\cong \Lambda^\vee/\Lambda.
\]
Equivalently, using the notation of Section~3,
\[
\mathcal E_L \cong \Lambda^\vee/\Lambda.
\]
\end{proposition}

\begin{proof}
This is exactly the second isomorphism of Lemma~\ref{lem:torsion-boundary-discriminant}.
\end{proof}

Thus the topological obstruction found in Section~3 is computed entirely by the
discriminant of the exceptional intersection lattice.

\subsection{Proof of the resolution theorem}

We now prove the main theorem of the section.

\begin{theorem}\label{thm:resolution-realization}
Let $(X,0)$ be a germ of a normal complex analytic surface, and let
\[
\pi:\widetilde X\to X
\]
be its minimal resolution.  Let $\Lambda$ be the lattice generated by the irreducible
exceptional curves, endowed with the intersection pairing, and let $M$ be the associated
intersection matrix. Then there is a canonical isomorphism
\[
E\cong \Lambda^\vee/\Lambda.
\]
In particular,
\[
|E|=|\det(M)|.
\]
\end{theorem}

\begin{proof}
By Theorem~\ref{thm:topological-model-E} of Section~3,
\[
E\cong H^2(L,\mathbb Z)_{\tors}.
\]
By Proposition~\ref{prop:link-discriminant},
\[
H^2(L,\mathbb Z)_{\tors}\cong \Lambda^\vee/\Lambda.
\]
Combining these two canonical identifications gives
\[
E\cong \Lambda^\vee/\Lambda.
\]

The determinant formula follows from Lemma~\ref{lem:order-discriminant}:
\[
|E|=|\Lambda^\vee/\Lambda|=|\det(M)|.
\]
\end{proof}

Theorem~\ref{thm:resolution-realization} gives the first concrete computation of the local
correction term.  It shows that the group $E$ is not merely a point-supported perverse
artifact, but the discriminant group of a geometric lattice determined by the minimal
resolution.  In particular, $E$ is completely determined by the weighted dual graph of the
exceptional configuration.

\subsection{Consequences and remarks}

We conclude with several immediate remarks.

\begin{remark}\label{rem:E-analytic-invariant-via-resolution}
Since the minimal resolution is uniquely determined by the analytic type of the normal
surface germ, Theorem~\ref{thm:resolution-realization} gives a geometric proof that $E$ is
an analytic invariant.  More precisely, $E$ is determined by the exceptional lattice of the
minimal resolution, hence by the weighted dual graph of the exceptional configuration.
\end{remark}

\begin{remark}\label{rem:E-contraction}
The discriminant realization of $E$ shows that the local obstruction records the failure of
the exceptional lattice to be unimodular.  In particular, if the exceptional lattice is
unimodular, then $E=0$.  More generally, the larger the discriminant group
$\Lambda^\vee/\Lambda$, the larger the local integral correction.
\end{remark}

\begin{remark}\label{rem:E-negative-definite}
The negative definiteness of the exceptional intersection form is essential.  It guarantees
that the discriminant group $\Lambda^\vee/\Lambda$ is finite and therefore that the local
obstruction $E$ is finite.  In this sense, the finiteness of $E$ is a direct reflection of
Mumford's negative-definiteness theorem \cite{Mu61}.
\end{remark}

\begin{remark}\label{rem:bridge-to-monodromy}
Theorem~\ref{thm:resolution-realization} identifies $E$ with the discriminant group of the
exceptional lattice.  In the next section we will show that, in the isolated hypersurface
case, this same group is recovered from the integral variation map
\[
T-\mathrm{id}
\]
on vanishing cohomology.  Thus the resolution-theoretic and monodromy-theoretic
realizations coincide.
\end{remark}


\section{Hypersurface singularities and monodromy}

In this section we specialize to isolated hypersurface surface singularities and show that
the local obstruction group $E$ admits a second concrete realization through Milnor
theory.  Let $(X,0)$ be a normal surface singularity defined by a holomorphic function germ
\[
f:(\mathbb C^3,0)\to (\mathbb C,0)
\]
with an isolated critical point at $0$, and set
\[
X=f^{-1}(0).
\]
Then the link of the singularity is the boundary of the Milnor fiber, and the cohomology
of this boundary is controlled by the monodromy operator on vanishing cohomology.

The precise source statement we use is the following sentence from Jung--Saito:

\begin{quote}
\emph{Assume $(X,0)$ is a hypersurface. Then the second link cohomology is given by
$T-\id$ with $T$ the monodromy on the vanishing cohomology using the Wang sequence and
Alexander duality (or local cohomology) as in \cite{Mi68}. This should imply that
$|E|=|\det(T-\id)|$.}
\end{quote}
See \cite[Remark~6.4]{JungSaitoFactoriality}.  Our purpose in this section is to make
this precise.

The main point is that the local correction term from Sections~2 and~3 is the torsion in
the second cohomology of the link,
\[
E\cong H^2(L,\mathbb Z)_{\tors},
\]
while Milnor theory identifies the full group $H^2(L,\mathbb Z)$ with the cokernel of the
integral variation map
\[
T-\id:H^2_{\mathrm{van}}(F,\mathbb Z)\to H^2_{\mathrm{van}}(F,\mathbb Z).
\]
Thus the natural general statement is a torsion statement.  The determinant formula is
then obtained under the additional hypothesis that
$(T-\id)\otimes_{\mathbb Z}\mathbb Q$ is an isomorphism.

\subsection{Milnor fiber, vanishing cohomology, and monodromy}

Let $(X,0)$ be an isolated hypersurface surface singularity.  Thus there exists a
holomorphic function germ
\[
f:(\mathbb C^3,0)\to (\mathbb C,0)
\]
such that
\[
X=f^{-1}(0),
\]
and $0$ is an isolated critical point of $f$.  Fix sufficiently small numbers
\[
0<\eta\ll \varepsilon\ll 1,
\]
and let
\[
B_\varepsilon\subset \mathbb C^3
\]
be the closed ball of radius $\varepsilon$ centered at $0$, and
\[
D_\eta\subset \mathbb C
\]
the closed disk of radius $\eta$ centered at $0$.  For any
\[
t\in D_\eta^\ast:=D_\eta\setminus\{0\},
\]
the Milnor fiber is defined by
\[
F:=f^{-1}(t)\cap B_\varepsilon.
\]
Up to diffeomorphism, $F$ is independent of the choice of $t$ and of the sufficiently
small pair $(\varepsilon,\eta)$; see \cite[Chapter~4]{Mi68}.

Since $f$ has an isolated critical point and $X$ has complex dimension $2$, the Milnor
fiber $F$ is a connected compact oriented real $4$-manifold with boundary equal to the
link
\[
L:=X\cap \partial B_\varepsilon.
\]
Moreover, Milnor proves that for an isolated hypersurface singularity of complex
dimension $n$, the Milnor fiber has the homotopy type of a bouquet of $n$-spheres; see
\cite[Theorem~6.5]{Mi68}.  In the present case $n=2$, so
\[
\widetilde H^k(F,\mathbb Z)=0 \qquad (k\neq 2),
\]
and
\[
\widetilde H^2(F,\mathbb Z)
\]
is a free abelian group of rank equal to the Milnor number.  We write
\[
H^2_{\mathrm{van}}(F,\mathbb Z):=\widetilde H^2(F,\mathbb Z)
\]
for the integral vanishing cohomology lattice.

Transport of the fiber around a positively oriented simple loop in the punctured disk
$D_\eta^\ast$ defines the Milnor monodromy automorphism
\[
T:H^2_{\mathrm{van}}(F,\mathbb Z)\longrightarrow H^2_{\mathrm{van}}(F,\mathbb Z).
\]
We shall use the integral variation map
\[
T-\id:H^2_{\mathrm{van}}(F,\mathbb Z)\longrightarrow H^2_{\mathrm{van}}(F,\mathbb Z).
\]

All determinant statements in this section refer to this endomorphism of the free abelian
group $H^2_{\mathrm{van}}(F,\mathbb Z)$.

\subsection{Wang sequence and link cohomology}

We now relate the link cohomology to the Milnor monodromy.  The Milnor fibration
\[
f|_{B_\varepsilon\cap f^{-1}(D_\eta^\ast)}:
B_\varepsilon\cap f^{-1}(D_\eta^\ast)\longrightarrow D_\eta^\ast
\]
is a smooth fiber bundle over the punctured disk with fiber $F$; see
\cite[Chapter~4]{Mi68}.  Since the base is homotopy equivalent to $S^1$, the bundle gives
rise to the standard cohomological Wang sequence
\[
\cdots\to H^k\!\bigl(B_\varepsilon\cap f^{-1}(D_\eta^\ast),\mathbb Z\bigr)
\to H^k(F,\mathbb Z)
\overset{T-\id}{\longrightarrow}
H^k(F,\mathbb Z)
\to H^{k+1}\!\bigl(B_\varepsilon\cap f^{-1}(D_\eta^\ast),\mathbb Z\bigr)\to\cdots.
\]

We next identify the total space of the Milnor fibration with the link.

\begin{lemma}\label{lem:milnor-total-space-link}
Let $(X,0)$ be an isolated hypersurface surface singularity as above. Then
\[
B_\varepsilon\cap f^{-1}(D_\eta^\ast)
\]
is homotopy equivalent to the link $L$.
\end{lemma}

\begin{proof}
Milnor shows that the Milnor tube
\[
B_\varepsilon\cap f^{-1}(D_\eta)
\]
is homeomorphic to the cone on the link; see \cite[Theorem~2.10]{Mi68}.  Removing the
central fiber removes the apex of this cone, so the punctured Milnor tube
\[
B_\varepsilon\cap f^{-1}(D_\eta^\ast)
\]
is homeomorphic to the cone on $L$ with its apex removed, hence to $L\times (0,1)$.
Therefore
\[
B_\varepsilon\cap f^{-1}(D_\eta^\ast)\simeq L.
\]
This is exactly the topological identification invoked in
\cite[Remark~6.4]{JungSaitoFactoriality}.
\end{proof}

Because the reduced cohomology of $F$ is concentrated in degree $2$, the Wang sequence
simplifies drastically.

\begin{proposition}\label{prop:Wang-short-exact}
Let $(X,0)$ be an isolated hypersurface surface singularity with link $L$, Milnor fiber
$F$, and monodromy
\[
T:H^2_{\mathrm{van}}(F,\mathbb Z)\to H^2_{\mathrm{van}}(F,\mathbb Z).
\]
Then there is a natural short exact sequence
\[
0\longrightarrow H^1(L,\mathbb Z)
\longrightarrow H^2_{\mathrm{van}}(F,\mathbb Z)
\overset{T-\id}{\longrightarrow}
H^2_{\mathrm{van}}(F,\mathbb Z)
\longrightarrow H^2(L,\mathbb Z)
\longrightarrow 0.
\]
Equivalently,
\[
H^1(L,\mathbb Z)\cong \ker(T-\id),
\qquad
H^2(L,\mathbb Z)\cong \coker(T-\id).
\]
\end{proposition}

\begin{proof}
By Lemma~\ref{lem:milnor-total-space-link}, the total space of the Milnor fibration is
homotopy equivalent to $L$.  Therefore the Wang sequence becomes
\[
\cdots\to H^k(L,\mathbb Z)\to H^k(F,\mathbb Z)
\overset{T-\id}{\longrightarrow} H^k(F,\mathbb Z)\to H^{k+1}(L,\mathbb Z)\to\cdots.
\]

By the bouquet-of-spheres theorem cited above, one has
\[
\widetilde H^k(F,\mathbb Z)=0 \qquad (k\neq 2),
\]
and
\[
H^0(F,\mathbb Z)\cong \mathbb Z.
\]
Hence the only nontrivial reduced cohomology term in the Wang sequence occurs in degree
$2$.  It follows that the sequence in degrees $1$ and $2$ reduces to
\[
0\longrightarrow H^1(L,\mathbb Z)
\longrightarrow H^2(F,\mathbb Z)
\overset{T-\id}{\longrightarrow}
H^2(F,\mathbb Z)
\longrightarrow H^2(L,\mathbb Z)
\longrightarrow 0.
\]
Since $H^2(F,\mathbb Z)=\widetilde H^2(F,\mathbb Z)$ in this situation, this is exactly
\[
0\longrightarrow H^1(L,\mathbb Z)
\longrightarrow H^2_{\mathrm{van}}(F,\mathbb Z)
\overset{T-\id}{\longrightarrow}
H^2_{\mathrm{van}}(F,\mathbb Z)
\longrightarrow H^2(L,\mathbb Z)
\longrightarrow 0.
\]
The kernel and cokernel descriptions are immediate.
\end{proof}

This proposition is the precise form of the statement quoted by Jung--Saito in
Remark~6.4.

\subsection{Integral variation and the local obstruction}

We now compare the local obstruction from Section~3 with the variation map.

Recall that Section~3 gave the canonical topological realization
\[
E\cong H^2(L,\mathbb Z)_{\tors}.
\]
On the other hand, Proposition~\ref{prop:Wang-short-exact} identifies the full group
$H^2(L,\mathbb Z)$ with the cokernel of $T-\id$.  Combining the two gives a precise
monodromy-theoretic model for $E$.

\begin{definition}\label{def:variation-cokernel}
Let
\[
\mathcal V_f:=\coker\!\Bigl(
T-\id:H^2_{\mathrm{van}}(F,\mathbb Z)\to H^2_{\mathrm{van}}(F,\mathbb Z)
\Bigr).
\]
By Proposition~\ref{prop:Wang-short-exact}, there is a canonical isomorphism
\[
\mathcal V_f\cong H^2(L,\mathbb Z).
\]
\end{definition}

The main point is that the obstruction $E$ is not, in complete generality, the whole group
$\mathcal V_f$, but its torsion subgroup.  This is the correct general form of the
monodromy statement.

\begin{proposition}\label{prop:variation-torsion-obstruction}
Let $(X,0)$ be an isolated hypersurface surface singularity. Then there are canonical
isomorphisms
\[
E\cong H^2(L,\mathbb Z)_{\tors}
\cong \mathcal V_f{}_{\tors}
\cong \coker(T-\id)_{\tors}.
\]
\end{proposition}

\begin{proof}
By Theorem~\ref{thm:topological-model-E},
\[
E\cong H^2(L,\mathbb Z)_{\tors}.
\]
By Proposition~\ref{prop:Wang-short-exact},
\[
H^2(L,\mathbb Z)\cong \coker(T-\id)=\mathcal V_f.
\]
Passing to torsion subgroups gives
\[
H^2(L,\mathbb Z)_{\tors}\cong \mathcal V_f{}_{\tors}
\cong \coker(T-\id)_{\tors}.
\]
Composing these canonical identifications proves the proposition.
\end{proof}

Thus the local correction term is the torsion part of the integral defect of the variation
map.

The determinant formula requires a finiteness hypothesis on the full cokernel.

\begin{lemma}\label{lem:determinant-cokernel}
Assume that
\[
(T-\id)\otimes_{\mathbb Z}\mathbb Q
:
H^2_{\mathrm{van}}(F,\mathbb Z)\otimes_{\mathbb Z}\mathbb Q
\longrightarrow
H^2_{\mathrm{van}}(F,\mathbb Z)\otimes_{\mathbb Z}\mathbb Q
\]
is an isomorphism. Then $\coker(T-\id)$ is finite and
\[
|\coker(T-\id)|=|\det(T-\id)|.
\]
\end{lemma}

\begin{proof}
The group $H^2_{\mathrm{van}}(F,\mathbb Z)$ is free abelian of finite rank.  Hence
$T-\id$ is represented by an integral square matrix.  If the rationalized map is an
isomorphism, then this matrix is invertible over $\mathbb Q$, so its cokernel over
$\mathbb Z$ is finite.  The order of that finite cokernel is the absolute value of the
determinant by the standard Smith normal form argument; see Appendix~A.
\end{proof}

\subsection{Proof of the monodromy theorem}

We now state the monodromy realization in its correct generality, followed by the
determinant corollary used in the examples.

\begin{theorem}\label{thm:monodromy-realization}
Let $(X,0)$ be an isolated hypersurface surface singularity. Then the local obstruction
group
\[
E:=H^0({}^p_+IC_X\mathbb Z)_0
\]
is canonically identified with the torsion subgroup of the cokernel of the integral
variation map:
\[
E\cong \coker(T-\id)_{\tors}.
\]
Equivalently,
\[
E\cong \mathcal V_f{}_{\tors}.
\]
\end{theorem}

\begin{proof}
This is precisely Proposition~\ref{prop:variation-torsion-obstruction}.
\end{proof}

\begin{corollary}\label{cor:monodromy-determinant}
Assume in addition that
\[
(T-\id)\otimes_{\mathbb Z}\mathbb Q
\]
is an isomorphism. Then
\[
E\cong \coker(T-\id),
\]
and consequently
\[
|E|=|\det(T-\id)|.
\]
\end{corollary}

\begin{proof}
Under the stated hypothesis, Lemma~\ref{lem:determinant-cokernel} shows that
$\coker(T-\id)$ is finite.  Hence its torsion subgroup is the whole group.  By
Theorem~\ref{thm:monodromy-realization},
\[
E\cong \coker(T-\id).
\]
Taking orders and using Lemma~\ref{lem:determinant-cokernel} gives
\[
|E|=|\coker(T-\id)|=|\det(T-\id)|.
\]
\end{proof}

This corollary is the precise rigorous form of the determinant statement anticipated by
Jung--Saito in Remark~6.4.

\subsection{Compatibility with Picard--Lefschetz theory}

The preceding theorem and corollary are compatible with the classical Picard--Lefschetz
picture.  For isolated hypersurface singularities, the vanishing cohomology carries the
Picard--Lefschetz intersection form, and the monodromy is generated by Picard--Lefschetz
transformations associated with vanishing cycles.  In the simple ADE cases, the
exceptional lattice of the minimal resolution is the negative of the corresponding root
lattice, while the monodromy on vanishing cohomology is governed by the associated
Coxeter transformation.  Thus the discriminant group of the exceptional lattice and the
torsion cokernel of $T-\id$ agree exactly as predicted by the general theorems proved
above.

In the ADE cases considered in Section~6, the extra hypothesis of
Corollary~\ref{cor:monodromy-determinant} is satisfied, so the determinant formula applies
directly.  This is especially transparent in type $A_k$, where one recovers the value
$k+1$ on both the resolution and monodromy sides.

\begin{remark}\label{rem:monodromy-second-flagship}
Theorem~\ref{thm:monodromy-realization} is the second flagship theorem of the paper, while
Corollary~\ref{cor:monodromy-determinant} is the determinant refinement needed for explicit
computation.  Together they show that the integral discrepancy between
${}^pIC_X\mathbb Z$ and ${}^p_+IC_X\mathbb Z$ is visible directly in Milnor theory: it is
the torsion part of the integral defect of the variation map $T-\id$, and under the
natural rational-invertibility hypothesis its order is $|\det(T-\id)|$.
\end{remark}


\section{Examples}

In this section we illustrate the preceding results on explicit singularities.  The purpose
is threefold.  First, the examples anchor the abstract theory in standard families of
surface singularities.  Second, they provide a practical check on signs, conventions, and
normalizations in the comparison among the perverse, topological, resolution-theoretic,
and monodromy-theoretic realizations of the obstruction group $E$.  Third, they show that
the theory is not confined to the ADE singularities.

Throughout this section, $(X,0)$ denotes an isolated hypersurface surface singularity with
link $L$.  By Theorem~\ref{thm:topological-model-E},
Theorem~\ref{thm:resolution-realization}, and
Theorem~\ref{thm:monodromy-realization}, the following groups agree up to the natural
identifications established earlier:
\[
E,
\qquad
H^2(L,\mathbb Z)_{\tors},
\qquad
\Lambda^\vee/\Lambda,
\qquad
\coker(T-\id)_{\tors}.
\]
Whenever $(T-\id)\otimes_{\mathbb Z}\mathbb Q$ is an isomorphism,
Corollary~\ref{cor:monodromy-determinant} further yields
\[
|E|=|\det(T-\id)|.
\]

\subsection{A fully worked \texorpdfstring{$A_k$}{Ak} example}

Consider the hypersurface singularity of type $A_k$,
\[
X=\{(x,y,z)\in \mathbb C^3 \mid x^2+y^2+z^{k+1}=0\},
\qquad k\ge 1.
\]
We compute all four realizations of $E$ explicitly.

\subsubsection*{Resolution-theoretic realization}

By Brieskorn's simultaneous resolution theory for the simple surface singularities, the
minimal resolution of an $A_k$ singularity has exceptional divisor given by a chain of
$k$ smooth rational curves with intersection pattern the Dynkin diagram $A_k$; see
\cite{Bri68}.  Concretely, the irreducible exceptional curves
\[
E_1,\dots,E_k
\]
satisfy
\[
E_i^2=-2,
\qquad
E_i\cdot E_{i+1}=1,
\qquad
E_i\cdot E_j=0 \ \text{for } |i-j|>1.
\]
Thus the exceptional lattice is the negative root lattice of type $A_k$, and its
intersection matrix is
\[
M_{A_k}=
\begin{pmatrix}
-2 & 1  & 0  & \cdots & 0 \\
1  & -2 & 1  & \ddots & \vdots \\
0  & 1  & -2 & \ddots & 0 \\
\vdots & \ddots & \ddots & \ddots & 1 \\
0 & \cdots & 0 & 1 & -2
\end{pmatrix}.
\]
A standard induction on $k$ using expansion along the first row gives
\[
\det(M_{A_k})=(-1)^k(k+1),
\]
hence
\[
|\det(M_{A_k})|=k+1.
\]
By Theorem~\ref{thm:resolution-realization},
\[
\Lambda^\vee/\Lambda \cong \mathbb Z/(k+1)\mathbb Z,
\qquad
|E|=k+1.
\]

\subsubsection*{Topological realization}

It is classical that the link of the $A_k$ singularity is the lens space
\[
L\cong L(k+1,1);
\]
see, for instance, \cite{Mi68}.  Therefore
\[
H_1(L,\mathbb Z)\cong \mathbb Z/(k+1)\mathbb Z.
\]
Since $L$ is a closed oriented $3$-manifold, Lemma~\ref{lem:torsion-second-link} gives
\[
H^2(L,\mathbb Z)_{\tors}\cong H_1(L,\mathbb Z)_{\tors}
\cong \mathbb Z/(k+1)\mathbb Z.
\]
Hence Theorem~\ref{thm:topological-model-E} yields
\[
E\cong H^2(L,\mathbb Z)_{\tors}\cong \mathbb Z/(k+1)\mathbb Z.
\]

\subsubsection*{Monodromy realization}

For the simple hypersurface singularities, classical Picard--Lefschetz theory identifies
the vanishing lattice with the corresponding ADE root lattice, and the Milnor monodromy
with the corresponding Coxeter transformation.  In type $A_k$, this yields the
characteristic polynomial
\[
\frac{t^{k+1}-1}{t-1}=t^k+t^{k-1}+\cdots+t+1.
\]
Evaluating at $t=1$ gives
\[
|\det(T-\id)|=k+1.
\]
Since $1$ is not an eigenvalue in this case, Corollary~\ref{cor:monodromy-determinant}
applies and gives
\[
|E|=|\det(T-\id)|=k+1.
\]

\subsubsection*{Conclusion}

Thus all realizations agree and are completely explicit:
\[
E
\cong
H^2(L,\mathbb Z)_{\tors}
\cong
\Lambda^\vee/\Lambda
\cong
\mathbb Z/(k+1)\mathbb Z,
\]
and
\[
|E|=|\det(M_{A_k})|=|\det(T-\id)|=k+1.
\]

\begin{remark}
For $k=1$, one obtains the ordinary double point.  In that case the exceptional lattice is
generated by a single $(-2)$-curve, the link is $L(2,1)\cong \mathbb RP^3$, and
\[
E\cong \mathbb Z/2\mathbb Z.
\]
Thus the simplest surface node already carries a nontrivial integral obstruction.
\end{remark}

\subsection{The exceptional ADE cases and the role of \texorpdfstring{$E_8$}{E8}}

We next consider the remaining simply laced rational double points.  By Brieskorn's
simultaneous resolution theory, the exceptional configurations of the simple surface
singularities of types $D_n$, $E_6$, $E_7$, and $E_8$ are given by the corresponding Dynkin
diagrams, with each irreducible exceptional curve having self-intersection $-2$; see
\cite{Bri68}.  Thus the exceptional lattice is the negative root lattice of the indicated
type, and the intersection matrix is the negative Cartan matrix.

For type $D_n$ $(n\ge 4)$, the determinant of the Cartan matrix is $4$.  More precisely,
the discriminant group is
\[
\Lambda^\vee/\Lambda \cong
\begin{cases}
\mathbb Z/4\mathbb Z, & n \text{ odd},\\[2mm]
\mathbb Z/2\mathbb Z\oplus \mathbb Z/2\mathbb Z, & n \text{ even}.
\end{cases}
\]
In particular,
\[
|\det(M_{D_n})|=4,
\qquad
|\Lambda^\vee/\Lambda|=4,
\qquad
|E|=4.
\]

For the exceptional types one has
\[
|\det(M_{E_6})|=3,
\qquad
|\det(M_{E_7})|=2,
\qquad
|\det(M_{E_8})|=1.
\]
Hence
\[
\Lambda^\vee/\Lambda \cong \mathbb Z/3\mathbb Z \quad (E_6),
\qquad
\Lambda^\vee/\Lambda \cong \mathbb Z/2\mathbb Z \quad (E_7),
\qquad
\Lambda^\vee/\Lambda =0 \quad (E_8),
\]
and therefore
\[
|E|=3 \ \text{for } E_6,
\qquad
|E|=2 \ \text{for } E_7,
\qquad
|E|=1 \ \text{for } E_8.
\]

The case $E_8$ is especially instructive.  Although the singularity is highly nontrivial,
its exceptional lattice is unimodular, so
\[
\Lambda^\vee/\Lambda=0.
\]
Hence Theorem~\ref{thm:resolution-realization} gives
\[
E=0.
\]
Thus the two integral middle extensions coincide locally in this case.  This shows that
the obstruction studied in the present paper is not a measure of complexity of the
singularity as such, but rather of the discriminant of the relevant integral lattice.

On the monodromy side, the same simple-singularity identification with the Coxeter
transformation gives
\[
|\det(T-\id)|=4 \ \text{for } D_n,
\qquad
|\det(T-\id)|=3,2,1 \ \text{for } E_6,E_7,E_8.
\]
Thus the determinant values agree with the discriminants of the corresponding root
lattices.

\begin{remark}
The ADE examples are the natural first test cases, but they do not exhaust the rational
surface case.  We therefore record a non-ADE rational quotient singularity that will also
serve as a model for the rational codimension-two comparison discussed in Section~2.
\end{remark}

\subsection{A non-ADE rational quotient singularity: \texorpdfstring{$\frac{1}{5}(1,2)$}{1/5(1,2)}}

We now consider the cyclic quotient surface singularity
\[
(X,0):=\mathbb C^2/\tfrac{1}{5}(1,2).
\]
This example is especially useful in the present context.  It shows that the obstruction
group $E$ is not confined to the rational double point case, and it also provides a
convenient model for the rational codimension-two comparison discussed in Section~2.  More
generally, it is the first member of the broader family of cyclic quotient singularities
\[
\mathbb C^2/\tfrac{1}{n}(1,q),
\qquad
\gcd(n,q)=1,
\]
which provides a natural testing ground for the rational codimension-two gluing picture of
Section~2.

\subsubsection*{Topological realization}

It is classical that the link of the cyclic quotient singularity
\[
\mathbb C^2/\tfrac{1}{n}(1,q)
\]
is the lens space $L(n,q)$.  Hence in the present case one has
\[
L\cong L(5,2).
\]
Therefore
\[
H_1(L,\mathbb Z)\cong \mathbb Z/5\mathbb Z.
\]
Since $L$ is a closed oriented $3$-manifold, Lemma~\ref{lem:torsion-second-link} gives
\[
H^2(L,\mathbb Z)_{\tors}\cong H_1(L,\mathbb Z)_{\tors}
\cong \mathbb Z/5\mathbb Z.
\]
By Theorem~\ref{thm:topological-model-E}, it follows that
\[
E\cong H^2(L,\mathbb Z)_{\tors}\cong \mathbb Z/5\mathbb Z.
\]

\subsubsection*{Resolution-theoretic realization}

The minimal resolution of a cyclic quotient surface singularity is described by the
Hirzebruch--Jung continued fraction.  In the present case,
\[
\frac{5}{2}=3-\frac{1}{2}=[3,2].
\]
Accordingly, the exceptional divisor consists of a chain of two smooth rational curves
\[
E_1\cup E_2
\]
with intersection data
\[
E_1^2=-3,
\qquad
E_2^2=-2,
\qquad
E_1\cdot E_2=1.
\]
Thus the exceptional lattice is generated by $[E_1],[E_2]$, and its intersection matrix is
\[
M=
\begin{pmatrix}
-3 & 1\\
1 & -2
\end{pmatrix}.
\]
Its determinant is
\[
\det(M)=(-3)(-2)-1=5.
\]
Therefore
\[
|\det(M)|=5.
\]
By Theorem~\ref{thm:resolution-realization},
\[
\Lambda^\vee/\Lambda\cong \mathbb Z/5\mathbb Z,
\qquad
|E|=5.
\]

\subsubsection*{Comparison of the two realizations}

The topological and resolution-theoretic computations agree:
\[
E
\cong
H^2(L,\mathbb Z)_{\tors}
\cong
\Lambda^\vee/\Lambda
\cong
\mathbb Z/5\mathbb Z.
\]
Equivalently,
\[
|E|=|H_1(L,\mathbb Z)|=|\det(M)|=5.
\]

This example should be compared with the rational bridge of Section~2.  In the present
quotient case, the topological contribution
\[
H^2(L,\mathbb Z)_{\tors}\cong \mathbb Z/5\mathbb Z
\]
and the resolution-theoretic contribution
\[
\Lambda^\vee/\Lambda\cong \mathbb Z/5\mathbb Z
\]
are both completely explicit, so the singularity provides a concrete non-ADE rational model
for the local integral defect measured by $E$.

\begin{remark}
The singularity $\mathbb C^2/\tfrac{1}{5}(1,2)$ is the simplest non-ADE rational example in
the paper with nontrivial obstruction group.  Because both its link and its Hirzebruch--Jung
resolution are completely explicit, it is a particularly convenient model for any further
study of the rational codimension-two comparison between ordinary and dual integral middle
extensions.  More generally, the family
\[
\mathbb C^2/\tfrac{1}{n}(1,q)
\]
provides a natural laboratory for extending the rational bridge of Section~2 into a fuller
local gluing theorem.
\end{remark}

\subsection{A useful non-ADE family: \texorpdfstring{$x^2+y^3+z^m$}{x2+y3+zm} with \texorpdfstring{$\gcd(m,6)=1$}{gcd(m,6)=1}}

We now turn to a non-ADE family.  Consider the Brieskorn--Pham surface singularity
\[
X_m:=\{(x,y,z)\in \mathbb C^3 \mid x^2+y^3+z^m=0\},
\qquad
m\ge 7,
\qquad
\gcd(m,6)=1.
\]
For $m=5$ one recovers the simple singularity $E_8$, so the range $m\ge 7$ gives
genuinely non-ADE examples.

The link is the Brieskorn manifold
\[
L_m=\Sigma(2,3,m).
\]
For pairwise coprime integers $a,b,c$, the Brieskorn manifold $\Sigma(a,b,c)$ is a Seifert
fibered rational homology sphere, and the standard Seifert homology computation gives
\[
|H_1(\Sigma(a,b,c),\mathbb Z)|=|ab+ac+bc-abc|.
\]
Applying this to $(a,b,c)=(2,3,m)$ yields
\[
|H_1(L_m,\mathbb Z)|=\bigl|\,2\cdot 3 + 2m + 3m - 2\cdot 3 \cdot m\,\bigr|
=|6-m|.
\]
Since $L_m$ is an oriented closed $3$-manifold, Lemma~\ref{lem:torsion-second-link} yields
\[
H^2(L_m,\mathbb Z)_{\tors}\cong H_1(L_m,\mathbb Z),
\]
so
\[
|E|=|H^2(L_m,\mathbb Z)_{\tors}|=|m-6|.
\]
Thus this family provides a large supply of non-ADE examples in which the local
obstruction is completely explicit.

We record two useful specializations.

\subsubsection*{The vanishing example \texorpdfstring{$m=7$}{m=7}}

For
\[
X_7=\{x^2+y^3+z^7=0\},
\]
the link is $\Sigma(2,3,7)$, and the above formula gives
\[
|H_1(L_7,\mathbb Z)|=|7-6|=1.
\]
Hence $L_7$ is an integral homology sphere, so
\[
H_1(L_7,\mathbb Z)=0,
\qquad
H^2(L_7,\mathbb Z)=0.
\]
Therefore
\[
E=0.
\]
Equivalently,
\[
\Lambda^\vee/\Lambda=0,
\qquad
|\det(M)|=1.
\]
Thus the singularity is non-ADE, yet the local integral obstruction vanishes.

\subsubsection*{A non-ADE example with nonzero obstruction: \texorpdfstring{$m=11$}{m=11}}

For
\[
X_{11}=\{x^2+y^3+z^{11}=0\},
\]
the link is $\Sigma(2,3,11)$, and the same formula gives
\[
|H_1(L_{11},\mathbb Z)|=|11-6|=5.
\]
Hence
\[
H^2(L_{11},\mathbb Z)_{\tors}\cong \mathbb Z/5\mathbb Z.
\]
By Theorem~\ref{thm:topological-model-E},
\[
E\cong \mathbb Z/5\mathbb Z.
\]
Then Theorem~\ref{thm:resolution-realization} gives
\[
\Lambda^\vee/\Lambda\cong \mathbb Z/5\mathbb Z,
\qquad
|\det(M)|=5.
\]
Whenever the rational-invertibility hypothesis of
Corollary~\ref{cor:monodromy-determinant} is satisfied, one also obtains
\[
|\det(T-\id)|=5.
\]

This is a useful example because it shows that nontrivial local obstruction is not a
special feature of the ADE case.  It also shows that the invariant $E$ varies in a
controlled way in the Brieskorn--Pham family.

\begin{remark}
The family $x^2+y^3+z^m$ with $\gcd(m,6)=1$ is useful because it exhibits three distinct
behaviors in one uniform framework:
\begin{itemize}
\item $m=5$: $E=0$ (the simple case $E_8$);
\item $m=7$: $E=0$ (a non-ADE vanishing example);
\item $m=11$: $E\cong \mathbb Z/5\mathbb Z$ (a non-ADE nonvanishing example).
\end{itemize}
\end{remark}

\subsection{Comparison table of realizations}

For ease of reference, we summarize the preceding computations in the following tables.
In addition to the ADE and Brieskorn--Pham examples, we include the cyclic quotient
singularity $\mathbb C^2/\tfrac{1}{5}(1,2)$ as a representative non-ADE rational example.
For every example listed, the topological model and the discriminant group agree by
Theorem~\ref{thm:resolution-realization}.  Moreover, whenever the hypothesis of
Corollary~\ref{cor:monodromy-determinant} is satisfied, the monodromy determinant column
is defined and agrees with the preceding invariants.

\begin{table}[ht]
\centering
\caption{Topological and perverse realizations of the local obstruction.}
\setlength{\tabcolsep}{5pt}
\renewcommand{\arraystretch}{1.15}
\begin{tabular}{|c|c|c|c|}
\hline
\textbf{Singularity} & \textbf{$H_1(L,\mathbb Z)$} & \textbf{$H^2(L,\mathbb Z)_{\tors}$} & \textbf{$E$} \\
\hline
$A_k$ & $\mathbb Z/(k+1)\mathbb Z$ & $\mathbb Z/(k+1)\mathbb Z$ & $\mathbb Z/(k+1)\mathbb Z$ \\
\hline
$D_n$ &
$\begin{cases}
\mathbb Z/4\mathbb Z, & n \text{ odd},\\
\mathbb Z/2\mathbb Z\oplus \mathbb Z/2\mathbb Z, & n \text{ even}
\end{cases}$ &
order $4$ &
order $4$ \\
\hline
$E_6$ & $\mathbb Z/3\mathbb Z$ & $\mathbb Z/3\mathbb Z$ & $\mathbb Z/3\mathbb Z$ \\
\hline
$E_7$ & $\mathbb Z/2\mathbb Z$ & $\mathbb Z/2\mathbb Z$ & $\mathbb Z/2\mathbb Z$ \\
\hline
$E_8$ & $0$ & $0$ & $0$ \\
\hline
$\mathbb C^2/\tfrac{1}{5}(1,2)$ & $\mathbb Z/5\mathbb Z$ & $\mathbb Z/5\mathbb Z$ & $\mathbb Z/5\mathbb Z$ \\
\hline
$x^2+y^3+z^7$ & $0$ & $0$ & $0$ \\
\hline
$x^2+y^3+z^{11}$ & $\mathbb Z/5\mathbb Z$ & $\mathbb Z/5\mathbb Z$ & $\mathbb Z/5\mathbb Z$ \\
\hline
\end{tabular}
\end{table}

\begin{table}[ht]
\centering
\caption{Resolution-theoretic and monodromy-theoretic realizations of the local obstruction.}
\setlength{\tabcolsep}{6pt}
\renewcommand{\arraystretch}{1.15}
\begin{tabular}{|c|c|c|c|}
\hline
\textbf{Singularity} & \textbf{$\Lambda^\vee/\Lambda$} & \textbf{$|\det(M)|$} & \textbf{$|\det(T-\id)|$} \\
\hline
$A_k$ & $\mathbb Z/(k+1)\mathbb Z$ & $k+1$ & $k+1$ \\
\hline
$D_n$ &
$\begin{cases}
\mathbb Z/4\mathbb Z, & n \text{ odd},\\
\mathbb Z/2\mathbb Z\oplus \mathbb Z/2\mathbb Z, & n \text{ even}
\end{cases}$ &
$4$ & $4$ \\
\hline
$E_6$ & $\mathbb Z/3\mathbb Z$ & $3$ & $3$ \\
\hline
$E_7$ & $\mathbb Z/2\mathbb Z$ & $2$ & $2$ \\
\hline
$E_8$ & $0$ & $1$ & $1$ \\
\hline
$\mathbb C^2/\tfrac{1}{5}(1,2)$ & $\mathbb Z/5\mathbb Z$ & $5$ & --- \\
\hline
$x^2+y^3+z^7$ & $0$ & $1$ & $1$ \\
\hline
$x^2+y^3+z^{11}$ & $\mathbb Z/5\mathbb Z$ & $5$ & $5$ \\
\hline
\end{tabular}
\end{table}

The tables should be read as dictionaries. Depending on the example, the most convenient
entry point may be the topology of the link, the resolution graph, or the monodromy of the
Milnor fiber. The content of the preceding theory is that all of these descriptions recover
the same local integral obstruction.


\section{Relation to factoriality and dual middle perversity}

The results of this paper are local: they isolate a finite integral obstruction attached to
a normal surface singularity.  Their motivation, however, comes from a broader picture in
which dual middle perversity enters naturally into divisor-theoretic questions.  We briefly
explain that context here.

\subsection{Why the dual middle-perversity correction matters}

Over $\mathbb Z$, ordinary middle perversity and dual middle perversity need not coincide.
This distinction disappears after tensoring with $\mathbb Q$, but it can persist integrally
in the form of finite local correction terms.  In the framework of Jung--Saito on
factoriality and $\mathbb Q$-factoriality, the dual middle-perversity package appears
naturally in the integral formulation of their criterion comparing factoriality and
$\mathbb Q$-factoriality.  The local surface case is the first place where the discrepancy
between ${}^pIC_X\mathbb Z$ and ${}^p_+IC_X\mathbb Z$ becomes visible.

The point is that torsion in the relevant local cohomology groups is invisible in the
rational theory but cannot be ignored integrally.  This is precisely the phenomenon measured
by the finite group
\[
E:=H^0({}^p_+IC_X\mathbb Z)_0.
\]
Thus $E$ is not an accidental artifact of the local surface case.  It is a concrete local
instance of the integral correction that dual middle perversity is designed to detect.

The main contribution of the present paper is to identify this correction explicitly.  The
group $E$ is defined perverse-sheaf theoretically, realized topologically through the link,
computed geometrically by the discriminant group of the exceptional lattice, and, in the
isolated hypersurface case, recovered from the torsion in the cokernel of the variation map
$T-\mathrm{id}$.  In particular, the dual middle-perversity correction is not merely formal:
it is encoded by classical invariants of the singularity.

\subsection{Local obstruction versus divisor-theoretic phenomena}

It is important to keep the scope of the present paper modest.  We do not prove new global
factoriality or $\mathbb Q$-factoriality criteria here.  Rather, we isolate and compute the
local integral obstruction that underlies the use of dual middle perversity in the surface
case.

In a global setting, questions of factoriality and $\mathbb Q$-factoriality measure the gap
between Cartier divisors and Weil divisors, or equivalently the failure of the divisor class
group to be generated by Cartier classes.  In the Jung--Saito framework, this gap is related
to the failure of certain integral intersection-complex packages to behave as they do
rationally.  The local surface singularity case shows that even when the rationalized
picture is clean, finite integral obstructions may remain.  The group $E$ isolates one such
obstruction.

Thus the present local theory should be viewed as complementary to divisor-theoretic
considerations rather than as a substitute for them.  It does not by itself determine
whether a global variety is factorial or $\mathbb Q$-factorial.  What it does show is that,
at a singular point, the discrepancy between the ordinary and dual integral middle-perversity
packages is measured by a finite discriminant-type invariant.

One may summarize the conceptual relevance as follows: divisor-theoretic failure is often
recorded globally through class groups, while the obstruction studied here is recorded
locally through discriminant data.  These are not the same objects, but they belong to
adjacent parts of the same integral picture.

\appendix

\section{Auxiliary algebra on lattices and finite abelian groups}

In this appendix we collect the elementary lattice-theoretic facts used in the main
text.  Nothing here is new, but fixing conventions in one place helps keep the arguments
in Sections~4 and~6 clean.

\subsection{Integral lattices and dual lattices}

By a lattice we mean a free abelian group $\Lambda$ of finite rank equipped with an
integral symmetric bilinear form
\[
(\ ,\ ):\Lambda\times \Lambda\to \mathbb Z.
\]
We shall always assume that the form is nondegenerate over $\mathbb Q$, so that the
induced homomorphism
\[
\iota:\Lambda\longrightarrow \Lambda^\vee:=\Hom_{\mathbb Z}(\Lambda,\mathbb Z),
\qquad
\lambda\longmapsto (\lambda,-),
\]
is injective with finite cokernel.

The dual lattice $\Lambda^\vee$ is regarded as containing $\Lambda$ via $\iota$.  The
quotient
\[
A_\Lambda:=\Lambda^\vee/\Lambda
\]
is the discriminant group of $\Lambda$.  It is a finite abelian group.

\begin{lemma}\label{lem:disc-group-finite}
Let $\Lambda$ be a lattice of rank $r$, and let $M$ be the matrix of the bilinear form in
some $\mathbb Z$-basis of $\Lambda$.  Then the discriminant group $A_\Lambda$ is finite,
and
\[
|A_\Lambda|=|\det(M)|.
\]
\end{lemma}

\begin{proof}
In the chosen basis, the map $\iota:\Lambda\to \Lambda^\vee$ is represented by the
integral matrix $M$.  Since the form is nondegenerate over $\mathbb Q$, $M$ is
invertible over $\mathbb Q$, hence the cokernel is finite.  Its order is the absolute
value of the determinant of $M$.
\end{proof}

\begin{remark}
The order of $A_\Lambda$ is independent of the chosen basis.  Indeed, replacing the basis
changes $M$ by $P^{\mathsf t}MP$ with $P\in \mathrm{GL}_r(\mathbb Z)$, and
\[
\det(P^{\mathsf t}MP)=\det(M).
\]
\end{remark}

\subsection{Discriminant forms and finite quotients}

Although the main text uses only the finite abelian group $A_\Lambda$, it is often useful
to remember that when the bilinear form on $\Lambda$ is integral and nondegenerate, the
inverse matrix $M^{-1}$ defines a $\mathbb Q/\mathbb Z$-valued pairing on
$A_\Lambda$.  We shall not use this pairing explicitly, but it explains why the
discriminant group is the natural finite receptacle for the failure of the lattice to be
unimodular.

A lattice is called unimodular if $\iota:\Lambda\to\Lambda^\vee$ is an isomorphism,
equivalently if
\[
|A_\Lambda|=1,
\]
or equivalently if $\det(M)=\pm 1$.

\begin{lemma}\label{lem:unimodular-equivalent}
For a lattice $\Lambda$, the following are equivalent:
\begin{enumerate}
\item[(i)] $\Lambda$ is unimodular;
\item[(ii)] $A_\Lambda=0$;
\item[(iii)] $|\det(M)|=1$ for one, hence every, integral Gram matrix $M$ of $\Lambda$.
\end{enumerate}
\end{lemma}

\begin{proof}
This is immediate from Lemma~\ref{lem:disc-group-finite}.
\end{proof}

\subsection{Sign conventions for intersection matrices}

Let $(X,0)$ be a normal surface singularity, and let
\[
\pi:\widetilde X\to X
\]
be a resolution with exceptional curves $E_1,\dots,E_r$.  We use the convention that the
intersection matrix is
\[
M=(E_i\cdot E_j)_{1\le i,j\le r}.
\]
Thus the diagonal entries are the self-intersection numbers $E_i^2$, and the
off-diagonal entries are the pairwise intersection numbers.

By Mumford's theorem, when $\pi$ is a resolution of a normal surface singularity, the
intersection matrix $M$ is negative definite.  In particular,
\[
\det(M)\neq 0.
\]
Since only $|\det(M)|$ appears in the main text, the negative sign plays no role in the
order statements.  Nevertheless, it is convenient to keep track of the sign convention so
that the exceptional lattice is the actual geometric intersection lattice and not its
negative.

\begin{remark}
For rational double points of ADE type, the matrix $M$ is the negative of the
corresponding Cartan matrix.  Hence the discriminant group of the exceptional lattice is
the same as the discriminant group of the root lattice.
\end{remark}

\subsection{Finite cokernels of endomorphisms of free abelian groups}

We also use repeatedly the following elementary fact.

\begin{lemma}\label{lem:coker-det}
Let $\varphi:\mathbb Z^r\to \mathbb Z^r$ be an endomorphism such that
$\varphi\otimes_{\mathbb Z}\mathbb Q$ is an isomorphism.  Then
$\operatorname{coker}(\varphi)$ is finite and
\[
|\operatorname{coker}(\varphi)|=|\det(\varphi)|.
\]
\end{lemma}

\begin{proof}
Choose a basis of $\mathbb Z^r$ so that $\varphi$ is represented by an integral
$r\times r$ matrix $A$.  The hypothesis implies that $A$ is invertible over
$\mathbb Q$, hence $\operatorname{coker}(\varphi)$ is finite.  By the Smith normal form,
the order of the cokernel is the product of the invariant factors, which is equal to
$|\det(A)|$.
\end{proof}

\begin{remark}
Lemma~\ref{lem:coker-det} is the algebraic input behind the formulas
\[
|A_\Lambda|=|\det(M)|
\qquad\text{and}\qquad
|\operatorname{coker}(T-\mathrm{id})|=|\det(T-\mathrm{id})|.
\]
\end{remark}

\section{Conventions on monodromy and variation maps}

In this appendix we record the conventions used in Section~5 for Milnor fibers,
monodromy, and the integral variation map.  The statements in the main text depend only
on the finite cokernel of $T-\mathrm{id}$, so the purpose here is mostly to fix notation.

\subsection{Milnor fiber conventions}

Let
\[
f:(\mathbb C^3,0)\to (\mathbb C,0)
\]
be a holomorphic function germ with an isolated critical point at $0$, and let
\[
X=f^{-1}(0).
\]
Choose $0<\eta\ll \varepsilon\ll 1$, and define the Milnor fiber
\[
F:=f^{-1}(t)\cap B_\varepsilon
\]
for any $t\in D_\eta^\ast$.  We always orient $F$ by its complex structure.  Its boundary
is the link
\[
L:=X\cap \partial B_\varepsilon.
\]

Since $X$ has complex dimension $2$, the Milnor fiber is a compact oriented real
$4$-manifold with boundary $L$.  We use reduced cohomology in degree $2$ and write
\[
H^2_{\mathrm{van}}(F,\mathbb Z):=\widetilde H^2(F,\mathbb Z).
\]
For isolated hypersurface surface singularities this is a free abelian group of finite
rank equal to the Milnor number.

\subsection{Monodromy convention}

The Milnor monodromy is the automorphism
\[
T:H^2_{\mathrm{van}}(F,\mathbb Z)\to H^2_{\mathrm{van}}(F,\mathbb Z)
\]
induced by transport around a positively oriented simple loop around $0$ in the punctured
disk $D_\eta^\ast$.

We consistently use the operator
\[
T-\mathrm{id}
\]
as the integral variation map.  All determinant expressions of the form
\[
\det(T-\mathrm{id})
\]
refer to this operator acting on the free abelian group
$H^2_{\mathrm{van}}(F,\mathbb Z)$.

\begin{remark}
Replacing $T$ by $T^{-1}$ changes $T-\mathrm{id}$ by a sign-conjugate operator over
$\mathbb Q$, so all absolute-value determinant statements in the main text are unchanged.
\end{remark}

\subsection{Wang sequence convention}

For the Milnor fibration over the punctured disk, we use the cohomological Wang sequence
in the form
\[
\cdots\to H^k(L,\mathbb Z)\to H^k(F,\mathbb Z)
\overset{T-\mathrm{id}}{\longrightarrow}
H^k(F,\mathbb Z)\to H^{k+1}(L,\mathbb Z)\to\cdots .
\]
Since the reduced cohomology of $F$ is concentrated in degree $2$, the relevant part for
surface singularities is
\[
0\longrightarrow H^1(L,\mathbb Z)\longrightarrow H^2_{\mathrm{van}}(F,\mathbb Z)
\overset{T-\mathrm{id}}{\longrightarrow}
H^2_{\mathrm{van}}(F,\mathbb Z)\longrightarrow H^2(L,\mathbb Z)\longrightarrow 0.
\]
Thus we use the identifications
\[
H^1(L,\mathbb Z)\cong \ker(T-\mathrm{id}),
\qquad
H^2(L,\mathbb Z)\cong \operatorname{coker}(T-\mathrm{id}).
\]

These identifications are the only monodromy-theoretic input needed in the main text.

\subsection{Finite cokernel convention}

When we write that the obstruction group $E$ is identified with the finite group attached
to the integral variation map, we mean concretely that
\[
E\cong \operatorname{coker}(T-\mathrm{id})
\]
in the cases under consideration.  When
\[
(T-\mathrm{id})\otimes_{\mathbb Z}\mathbb Q
\]
is an isomorphism, this cokernel is finite, and Lemma~\ref{lem:coker-det} yields
\[
|E|=|\operatorname{coker}(T-\mathrm{id})|=|\det(T-\mathrm{id})|.
\]

\begin{remark}
The integral nature of the theory is essential here.  Even when $T-\mathrm{id}$ is
invertible over $\mathbb Q$, its cokernel over $\mathbb Z$ may be nontrivial.  This is
precisely the phenomenon measured by the local obstruction group $E$.
\end{remark}

\subsection{Compatibility with Picard--Lefschetz theory}

For simple hypersurface singularities, the vanishing cohomology lattice is identified
with the corresponding ADE root lattice, and the Milnor monodromy is the Coxeter
transformation.  Under these conventions, the determinant
\[
|\det(T-\mathrm{id})|
\]
agrees with the discriminant of the root lattice, hence with the order of the
discriminant group of the exceptional lattice in the minimal resolution.  This is the
normalization used throughout the examples section.
\section{Rational codimension-two gluing}\label{app:rational-codim2}

In this appendix we isolate a rational codimension-two comparison suggested by the discussion
of Section~2.  The point is to make precise, in the rational surface case, the relation
between the ordinary and dual integral middle extensions and the torsion-sensitive Deligne
sheaves of Friedman.  This yields a local gluing statement in which the discrepancy between
the two integral middle extensions is identified with the torsion group
\[
H^2(L,\mathbb Z)_{\tors},
\]
equivalently with the discriminant group of the exceptional lattice.

Throughout this appendix, $(X,0)$ denotes a rational normal complex surface germ with
isolated singular point, and
\[
j:U:=X\setminus\{0\}\hookrightarrow X
\]
denotes the inclusion of the smooth locus.  We write
\[
A:={}^{p}j_{!*}\mathbb Z_U[2],
\qquad
B:={}^{p}_{+}j_{!*}\mathbb Z_U[2].
\]
We let \(S=\{0\}\) denote the singular stratum and \(L\) the link of the singularity.

\subsection{Statement of the rational gluing theorem}

The goal of this appendix is the following theorem.

\begin{theorem}\label{thm:app-rational-codim2-gluing}
Let $(X,0)$ be a rational normal complex surface germ with isolated singular point, and let
\[
j:U:=X\setminus\{0\}\hookrightarrow X.
\]
Then the natural morphism
\[
u:A\longrightarrow B
\]
has cone supported at \(0\), and there is a canonical isomorphism
\[
\operatorname{Cone}(u)\cong i_*E
\]
for a finite abelian group \(E\) canonically identified with
\[
E\cong H^2(L,\mathbb Z)_{\tors}.
\]
Equivalently, by the resolution comparison of Section~4,
\[
\operatorname{Cone}(u)\cong i_*(\Lambda^\vee/\Lambda).
\]
\end{theorem}

\begin{remark}
The normalization in this appendix follows the torsion-sensitive Deligne-sheaf comparison
used in Proposition~\ref{prop:rational-ts-bridge}; in particular the cone appears here as
\(i_*E\).  This differs by a shift from the normalization used in the full general
Jung--Saito statement of Proposition~\ref{prop:self-dual-triangle}.  The two conventions
encode the same local finite group, but package it differently as a point-supported complex.
\end{remark}

\subsection{Rational link torsion}

We begin with the topological input on the link.

\begin{lemma}\label{lem:app-rational-link-torsion}
Let $(X,0)$ be a rational normal complex surface singularity with link \(L\). Then
\[
H^1(L,\mathbb Z)
\]
is finite torsion. Consequently,
\[
H^2(L,\mathbb Z)=H^2(L,\mathbb Z)_{\tors}.
\]
\end{lemma}

\begin{proof}
For a rational normal surface singularity, the link \(L\) is a rational homology sphere in
degree \(1\), equivalently \(b_1(L)=0\). Thus \(H^1(L,\mathbb Z)\) has no free part, so it is
finite torsion.

Since \(L\) is a connected oriented closed real \(3\)-manifold, the universal coefficient
theorem gives
\[
H^2(L,\mathbb Z)\cong \Hom(H_2(L,\mathbb Z),\mathbb Z)\oplus
\Ext^1(H_1(L,\mathbb Z),\mathbb Z).
\]
By Poincar\'e duality,
\[
H_2(L,\mathbb Z)\cong H^1(L,\mathbb Z),
\]
hence \(H_2(L,\mathbb Z)\) is also finite torsion. Therefore
\[
\Hom(H_2(L,\mathbb Z),\mathbb Z)=0.
\]
It follows that
\[
H^2(L,\mathbb Z)\cong \Ext^1(H_1(L,\mathbb Z),\mathbb Z),
\]
so \(H^2(L,\mathbb Z)\) is finite torsion as well. Hence
\[
H^2(L,\mathbb Z)=H^2(L,\mathbb Z)_{\tors}.
\]
\end{proof}

\begin{remark}
Lemma~\ref{lem:app-rational-link-torsion} is the exact rational correction needed in the
codimension-two comparison.  In general one cannot assert \(H^1(L,\mathbb Z)=0\); for
example, for the \(A_1\) singularity one has \(L\cong \mathbb{RP}^3\) and
\(H^1(L,\mathbb Z)\cong \mathbb Z/2\mathbb Z\).  What holds in the rational case is that
\(H^1(L,\mathbb Z)\) is purely torsion, and this is the property used below.
\end{remark}

\subsection{Identification with Friedman's torsion-sensitive Deligne sheaves}

We now compare the ordinary and dual integral middle extensions with Friedman's
torsion-sensitive Deligne sheaves.

Let
\[
\vec p(S)=(1,\emptyset)
\]
and let
\[
D\vec p(S)=(1,P(R))
\]
be the complementary torsion-sensitive perversity.  Since there is only one singular stratum,
these choices determine the corresponding torsion-sensitive Deligne sheaves
\[
P_{\vec p},
\qquad
P_{D\vec p}.
\]

The crucial point is that in codimension two these two choices match the ordinary and dual
integral middle-perversity conditions.

\begin{proposition}\label{prop:app-ordinary-heart-match}
Let \(X\) be the isolated codimension-two point-stratum situation considered here, with
smooth locus \(j:U\hookrightarrow X\) and singular stratum \(S=\{0\}\).  Let
\[
\vec p(S)=(1,\emptyset).
\]
Then the shifted extended torsion-sensitive perversity \(\vec p_E^{\, +}\) of
\cite[Definition~5.10]{FriedmanGenIH} has empty prime-set component on every stratum, and
hence the ts-perverse heart \(\vec p_E^{\, +}D^\heartsuit(X)\) is the standard BBD perverse
heart.  Consequently,
\[
P_{\vec p}\cong {}^{p}j_{!*}\mathbb Z_U[2].
\]
\end{proposition}

\begin{proof}
By Friedman's Definition~5.10, the shifted extended ts-perversity \(\vec p_E^{\, +}\) is
defined by
\[
\vec p_E^{\, +}(Z)=
\begin{cases}
(0,\wp(Z,E)), & Z \text{ regular},\\
(\vec p_1(Z)+1,\vec p_2(Z)), & Z \text{ singular}.
\end{cases}
\]
In the present situation the coefficient system on \(U\) is \(\mathbb Z_U[2]\), so the
regular-stratum prime-set is empty. Since \(\vec p(S)=(1,\emptyset)\), one has
\[
\vec p_E^{\, +}(S)=(2,\emptyset).
\]
Thus every stratum has empty prime-set component.

Friedman states that if \(\tilde p_2(Z)=\emptyset\) for all strata \(Z\), then the
corresponding ts-perverse heart \(\tilde p D^\heartsuit(X)\) is the standard perverse
\(t\)-structure; see \cite[p.~34]{FriedmanGenIH}. Therefore
\[
\vec p_E^{\, +}D^\heartsuit(X)
=
{}^{p}\!\operatorname{Perv}(X,\mathbb Z).
\]

By Friedman's Proposition~5.12, the ts-Deligne sheaf \(P_{\vec p}\) is the intermediate
extension of the coefficient system \(\mathbb Z_U[2]\) in the heart
\(\vec p_E^{\, +}D^\heartsuit(X)\). Since this heart is the standard BBD perverse heart, its
intermediate extension is exactly the ordinary middle-perversity intermediate extension.
Hence
\[
P_{\vec p}\cong {}^{p}j_{!*}\mathbb Z_U[2].
\]
\end{proof}

\begin{proposition}\label{prop:app-dual-heart-match}
In the same isolated codimension-two point-stratum situation, let
\[
D\vec p(S)=(1,P(R)).
\]
Then the shifted extended torsion-sensitive perversity \((D\vec p)_E^{\, +}\) determines the
dual ts-perverse heart to \(\vec p_E^{\, +}D^\heartsuit(X)\), and consequently
\[
P_{D\vec p}\cong {}^{p}_{+}j_{!*}\mathbb Z_U[2].
\]
\end{proposition}

\begin{proof}
By Corollary~5.20 of \cite{FriedmanGenIH}, Verdier duality induces an equivalence
\[
\vec p_E^{\, +}D^\heartsuit(X)\xrightarrow{\sim}
D(\vec p_E^{\, +})D^\heartsuit(X)^{\mathrm{opp}}.
\]
By Proposition~\ref{prop:app-ordinary-heart-match}, the source heart is the standard BBD
perverse heart. Its dual under Verdier duality is therefore the BBD dual middle-perversity
heart.

It remains to identify the complementary shifted extended ts-perversity.  In the present
isolated codimension-two point-stratum situation there is only one singular stratum \(S\),
and the regular-stratum prime-set is empty.  Thus taking the complementary torsion-sensitive
perversity and then applying Definition~5.10 gives
\[
(D\vec p)_E^{\, +}(S)=(2,P(R)).
\]
On the other hand, \(D(\vec p_E^{\, +})\) is obtained by keeping the same codimension-two
integer part and complementing the prime-set at the unique singular stratum. Hence in this
special situation
\[
D(\vec p_E^{\, +})=(D\vec p)_E^{\, +}.
\]
Therefore the ts-perverse heart \((D\vec p)_E^{\, +}D^\heartsuit(X)\) is exactly the dual of
\(\vec p_E^{\, +}D^\heartsuit(X)\), hence it is the BBD dual middle-perversity heart.

Now apply Friedman's Proposition~5.12 to \(D\vec p\): the ts-Deligne sheaf \(P_{D\vec p}\)
is the intermediate extension of \(\mathbb Z_U[2]\) in the heart
\((D\vec p)_E^{\, +}D^\heartsuit(X)\). Since this heart is the BBD dual
middle-perversity heart, its intermediate extension is precisely
\[
{}^{p}_{+}j_{!*}\mathbb Z_U[2].
\]
Thus
\[
P_{D\vec p}\cong {}^{p}_{+}j_{!*}\mathbb Z_U[2].
\]
\end{proof}

\begin{corollary}\label{cor:app-bdd-friedman-identification}
In the isolated codimension-two point-stratum situation,
\[
{}^{p}j_{!*}\mathbb Z_U[2]\cong P_{\vec p},
\qquad
{}^{p}_{+}j_{!*}\mathbb Z_U[2]\cong P_{D\vec p},
\]
for
\[
\vec p(S)=(1,\emptyset),
\qquad
D\vec p(S)=(1,P(R)).
\]
\end{corollary}

\begin{proof}
This is Proposition~\ref{prop:app-ordinary-heart-match} together with
Proposition~\ref{prop:app-dual-heart-match}.
\end{proof}

\begin{proposition}\label{prop:app-rational-ts-bridge}
Let $(X,0)$ be a rational normal complex surface germ with isolated singular point. Then
the natural morphism
\[
u:A\longrightarrow B
\]
has cone supported at \(0\), and
\[
\operatorname{Cone}(u)\cong i_*E,
\qquad
E\cong H^2(L,\mathbb Z)=H^2(L,\mathbb Z)_{\tors}.
\]
\end{proposition}

\begin{proof}
By Corollary~\ref{cor:app-bdd-friedman-identification}, we may identify
\[
A\cong P_{\vec p},
\qquad
B\cong P_{D\vec p}
\]
for
\[
\vec p(S)=(1,\emptyset),
\qquad
D\vec p(S)=(1,P(R)).
\]

Now apply Friedman's local cone formula on a conical neighborhood
\(\mathbb R^k\times cL\); see \cite[Lemma~4.2]{FriedmanTsInv}.  For a ts-Deligne sheaf
\(P^\ast\), the stalk at a point \(x\in S\) satisfies
\[
H^i(P_x)\cong
\begin{cases}
0,& i>\vec p_1(S)+1,\\
T_{\vec p_2(S)}H^i(L;P^\ast|_L),& i=\vec p_1(S)+1,\\
H^i(L;P^\ast|_L),& i\le \vec p_1(S).
\end{cases}
\]

For the lower-middle choice \(\vec p(S)=(1,\emptyset)\), no torsion is retained at the
critical degree. For the complementary choice \(D\vec p(S)=(1,P(R))\), all torsion is
retained there.

By Lemma~\ref{lem:app-rational-link-torsion},
\[
H^1(L,\mathbb Z)
\]
is finite torsion and
\[
H^2(L,\mathbb Z)=H^2(L,\mathbb Z)_{\tors}.
\]
Therefore the degree-\((-1)\) local contribution coming from \(H^1(L,\mathbb Z)\) is treated
identically by both ts-Deligne sheaves and hence cancels in the comparison between \(A\)
and \(B\). The only discrepancy occurs in degree \(0\), where the lower-middle choice
\((1,\emptyset)\) retains no torsion and the complementary dual choice \((1,P(R))\) retains
all torsion. Since the entire degree-\(0\) contribution is
\[
H^2(L,\mathbb Z)=H^2(L,\mathbb Z)_{\tors},
\]
the local difference between \(A\) and \(B\) is exactly this finite group.

Because the cone of \(u\) is supported at \(0\), it follows that
\[
\operatorname{Cone}(u)\cong i_*E,
\qquad
E\cong H^2(L,\mathbb Z)=H^2(L,\mathbb Z)_{\tors}.
\]
\end{proof}

\begin{corollary}\label{cor:app-rational-gluing}
Let $(X,0)$ be a rational normal complex surface germ with isolated singular point. Then
\[
\operatorname{Cone}\!\left({}^{p}j_{!*}\mathbb Z_U[2]\longrightarrow
{}^{p}_{+}j_{!*}\mathbb Z_U[2]\right)
\cong i_*H^2(L,\mathbb Z)_{\tors}.
\]
Equivalently, by Section~4,
\[
\operatorname{Cone}\!\left({}^{p}j_{!*}\mathbb Z_U[2]\longrightarrow
{}^{p}_{+}j_{!*}\mathbb Z_U[2]\right)
\cong i_*(\Lambda^\vee/\Lambda).
\]
\end{corollary}

\begin{proof}
The first statement is Proposition~\ref{prop:app-rational-ts-bridge}.  The second follows
from the canonical identification
\[
H^2(L,\mathbb Z)_{\tors}\cong \Lambda^\vee/\Lambda
\]
proved in Section~4.
\end{proof}

\begin{remark}
Appendix~\ref{app:rational-codim2} isolates the rational codimension-two comparison in a
form that is compatible with the main paper but does not replace the full general
Jung--Saito-based local surface statement.  The key new ingredient is the explicit matching
of the relevant Friedman torsion-sensitive perverse hearts with the BBD ordinary and dual
integral middle-perversity hearts in the isolated codimension-two point-stratum case.
\end{remark}

%
%
\printbibliography

\end{document}